\documentclass[10pt,reqno]{amsart}

\usepackage{amsmath, amssymb, amsthm}
\usepackage{geometry}
\newtheorem{theorem}{Theorem}
\newtheorem{lemma}[theorem]{Lemma}

\theoremstyle{definition}
\newtheorem{definition}[theorem]{Definition}

\begin{document}
	\title[Asymptotics of Second order Linear ODEs]
	{Characterisation of Stability and Decay Rates in a Weakly Damped Second Order Linear Differential Equation}
	
	\author{John A. D. Appleby}
	\address{School of Mathematical
		Sciences, Dublin City University, Glasnevin, Dublin 9, Ireland}
	\email{john.appleby@dcu.ie} 
	
	\author{Subham Pal}
	\address{School of Mathematical
		Sciences, Dublin City University, Glasnevin, Dublin 9, Ireland}
	\email{subham.pal2@mail.dcu.ie}
	
	\date{23 March 2026}
	
	\begin{abstract}
		This paper gives necessary and sufficient conditions for the convergence of the solution of a weakly damped second order linear differential equation that is subjected to outside forcing, for which solutions of the unforced equation are asymptotically stable. Conditions are also given which characterise when the solution and its derivative tend to zero. Finally, we give sharp sufficient conditions under which the solution of the forced equation has the same asymptotic behaviour as the unforced equation, to leading order.  
	\end{abstract}
	
	\maketitle
	
	\section{Introduction}	
	We study the following second-order linear ordinary differential equation:
	\begin{equation} \label{eq:original}
		x''(t) + p(t)x'(t) + \omega^2 x(t) = f(t), \quad t \ge 0,
	\end{equation}
	subject to the initial conditions $x(0) = \xi_0$ and $x'(0) = \xi_1$.
	\begin{definition}[Hypotheses on the Coefficients]
		We impose the following strict conditions on the damping coefficient $p(t)$ and the forcing function $f(t)$:
		\begin{enumerate}
			\item $p \in C^1([0, \infty))$.
			\item Positivity and monotonicity: $p(t) > 0$ and $p'(t) < 0$ for all $t \ge 0$.
			\item Non-integrability of $p$: $p \notin L^1(0, \infty)$, meaning $\int_0^\infty p(t) \, dt = \infty$.
			\item Integrability of $p^2$: $p^2 \in L^1(0, \infty)$, meaning $\int_0^\infty p(t)^2 \, dt < \infty$.
			\item Local integrability of forcing: $f \in L^1_{\text{loc}}([0, \infty))$.
		\end{enumerate}
	\end{definition}
	As can be seen from the hypotheses on $p$, the perturbed Bessel equation 
	\[
	x''(t)+\frac{1}{t}x'(t)+x(t)=f(t),
	\]
	for $t\geq 1$ is covered by this class, where $p(t)=1/t$ and $\omega=1$. 
	
	The purpose of this paper is the following. Under the above conditions on $p$, it is known that the solution of 
	the unperturbed equation 
	\[
	u''(t)+p(t)u(t)+\omega^2 u(t)=0, \quad t\geq 0
	\]
	obeys $u(t)\to 0$ and $u'(t)\to 0$ as $t\to\infty$. Moreover, the leading order asymptotic behaviour of the solution is given by 
	\[
	u(t)=e^{-\frac{1}{2}\int_0^t p(s)\,ds} \left(c_1 \cos(\omega t) + c_2 \sin(\omega t)\right)+o\left(e^{-\frac{1}{2}\int_0^t p(s)\,ds} \right), \quad t\to\infty.
	\]
	Therefore, it is tempting to believe that, if $f$ tends to zero (or does so rapidly enough), these properties will be preserved by the solutions of $x$. In this paper, we use the auxiliary functions $y_1$ and $y_2$ defined by 
	\[
	y_1(t)=\int_0^t e^{-\omega^2 s} f(t-s)\,ds, \quad y_2(t)=\int_0^t se^{-\omega^2 s} f(t-s)\,ds
	\]
	to characterise the asymptotic behaviour of $x$ and $x'$. 
	Specifically, if 
	\[
	\rho(t)=\exp\left(-\frac{1}{2}\int_0^t p(s)\,ds\right), \quad t\geq 0,
	\]
	we prove that $x(t)\to 0$ as $t\to\infty$ if and only if 
	\[
	y_2(t)\to 0, \quad \int_0^t \frac{1}{\rho(s)} \sin(\omega s)y_2(s)\,ds = o\left(\frac{1}{\rho(t)}\right), \quad  \int_0^t \frac{1}{\rho(s)} \cos(\omega s)y_2(s)\,ds = o\left(\frac{1}{\rho(t)}\right), \quad t\to\infty,
	\]
	that $x(t)\to 0$ and $x'(t)\to 0$ as $t\to\infty$ if and only 
		\[
	y_1(t)\to 0, \quad \int_0^t \frac{1}{\rho(s)} \sin(\omega s)y_1(s)\,ds = o\left(\frac{1}{\rho(t)}\right), \quad  \int_0^t \frac{1}{\rho(s)} \cos(\omega s)y_1(s)\,ds = o\left(\frac{1}{\rho(t)}\right), \quad t\to\infty.
	\]
	Therefore, we have a characterisation of the convergence of $x$ and $x'$ to zero which can be given in terms of  elementary functions. Moreover, we can give sharp sufficient conditions on $y_2$ under which the leading order asymptotic behaviour of $u$ is preserved by $x$. In particular, we have that 
	\[
	\int_0^\infty \frac{1}{\rho(s)}|y_2(s)|\,ds <+\infty,\quad y_2(t)=o(1/\rho(t)), \quad t\to\infty,
	\]
	is sufficient to give 
	\[
		x(t)=e^{-\frac{1}{2}\int_0^t p(s)\,ds} \left(c_1 \cos(\omega t) + c_2 \sin(\omega t)\right)+o\left(e^{-\frac{1}{2}\int_0^t p(s)\,ds} \right), \quad t\to\infty.
	\]
	The sharpness of these conditions can be seen by noting that if $x$ has the above leading order behaviour, $y_2$ must necessarily behave according to  
	\[
	\lim_{t\to\infty}
	\int_0^t \frac{1}{\rho(s)}\sin(\omega s)y_2(s)\,ds \quad 
		\lim_{t\to\infty}
	\int_0^t \frac{1}{\rho(s)}\cos\omega s)y_2(s)\,ds \text{ are finite, and }
	y_2(t)=o(1/\rho(t)), \quad t\to\infty.
	\]	
The idea of characterising asymptotic behaviour in differential systems by averaged conditions of the type used here goes back to 
Strauss and Yorke \cite{SY67b,SY67a}, and Gripenberg, Londen and Staffans \cite{GLS}. Recently, Appleby and Lawless, in a sequence of papers, have shown how these results can be extended to stochastic systems and systems with memory \cite{AL:2023(AppliedMathLetters)} \cite{AL:2023(AppliedNumMath)}, \cite{AL:2024}, \cite{AL:panto24}.

	\section{Transformation to the Unperturbed Form}
	
	To eliminate the first-derivative term in \eqref{eq:original}, we apply a standard change of variables.
	
	\begin{lemma}[Elimination of Damping]
		Define the function $u(t)$ via the transformation:
		\begin{equation} \label{eq:transform}
			x(t) = \exp\left(-\frac{1}{2}\int_0^t p(s) \, ds\right) u(t).
		\end{equation}
		Then $u(t)$ satisfies the differential equation:
		\begin{equation} \label{eq:ueq}
			u''(t) + (\omega^2 + q(t))u(t) = g(t),
		\end{equation}
		where the perturbation $q(t)$ and the new forcing $g(t)$ are given by:
		\begin{align}
			q(t) &= -\frac{1}{4}p(t)^2 - \frac{1}{2}p'(t), \\
			g(t) &= f(t)\exp\left(\frac{1}{2}\int_0^t p(s) \, ds\right).
		\end{align}
	\end{lemma}
	
	\begin{proof}
		By the Fundamental Theorem of Calculus, differentiating $x(t)$ using the product and chain rules yields:
		\begin{align*}
			x'(t) &= e^{-\frac{1}{2}\int_0^t p(s)ds} \left( u'(t) - \frac{1}{2}p(t)u(t) \right), \\
			x''(t) &= e^{-\frac{1}{2}\int_0^t p(s)ds} \left[ u''(t) - p(t)u'(t) + \left(\frac{1}{4}p(t)^2 - \frac{1}{2}p'(t)\right)u(t) \right].
		\end{align*}
		Substitute $x(t)$, $x'(t)$, and $x''(t)$ into the original equation \eqref{eq:original}:
		\begin{align*}
			e^{-\frac{1}{2}\int_0^t p(s)ds} \Bigg[ &u''(t) - p(t)u'(t) + \left(\frac{1}{4}p(t)^2 - \frac{1}{2}p'(t)\right)u(t) \\
			&+ p(t)\left(u'(t) - \frac{1}{2}p(t)u(t)\right) + \omega^2 u(t) \Bigg] = f(t).
		\end{align*}
		Notice that the terms $-p(t)u'(t)$ and $+p(t)u'(t)$ cancel exactly. Grouping the coefficients of $u(t)$ and multiplying both sides by $\exp(\frac{1}{2}\int_0^t p(s)ds)$ directly yields equation \eqref{eq:ueq}.
	\end{proof}
	
	\section{Integrability of the Perturbation}
	
	For asymptotic perturbation theory to apply, the perturbation term $q(t)$ must vanish sufficiently fast. 
	
	\begin{theorem}[$L^1$ Integrability of $q$]
		Under the stated hypotheses, $q \in L^1(0, \infty)$.
	\end{theorem}
	
	\begin{proof}
		We must show that $\int_0^\infty |q(t)| \, dt < \infty$. By the triangle inequality:
		\[ \int_0^\infty |q(t)| \, dt \le \frac{1}{4}\int_0^\infty p(t)^2 \, dt + \frac{1}{2}\int_0^\infty |p'(t)| \, dt. \]
		By hypothesis 4, $p^2 \in L^1(0, \infty)$, so the first term is finite. 
		
		To evaluate the second term, we analyze $p(t)$. By hypothesis 2, $p(t) > 0$ and $p'(t) < 0$. Since $p(t)$ is strictly decreasing and bounded below by zero, the Monotone Convergence Theorem implies that $\lim_{t \to \infty} p(t) = L$ exists and $L \ge 0$. 
		
		We claim $L = 0$. Suppose $L > 0$. Then $p(t) \ge L$ for all $t$, which implies $p(t)^2 \ge L^2 > 0$. This would mean $\int_0^\infty p(t)^2 \, dt = \infty$, violating hypothesis 4. Thus, $L = 0$.
		
		Because $p'(t) < 0$, we have $|p'(t)| = -p'(t)$. Evaluating the improper integral gives:
		\[ \int_0^\infty |p'(t)| \, dt = \lim_{T \to \infty} \int_0^T -p'(t) \, dt = \lim_{T \to \infty} (p(0) - p(T)) = p(0) - 0 = p(0) < \infty. \]
		Since both components of $q(t)$ are absolutely integrable, $q \in L^1(0, \infty)$.
	\end{proof}
	
	\section{Variation of Constants and Resolvents}
	
	We now construct the solutions using the Green's function (differential resolvent).
	
	\begin{theorem}[Variation of Constants]
		Let $u_1(t; s)$ and $u_2(t; s)$ be the fundamental solutions to the homogeneous equation $u'' + (\omega^2 + q(t))u = 0$ satisfying the initial conditions at $t=s$: $u_1(s; s)=1, u_1'(s; s)=0$ and $u_2(s; s)=0, u_2'(s; s)=1$. 
		
		Define the differential resolvent for the $u$-equation as $G_u(t, s) = u_1(s)u_2(t) - u_1(t)u_2(s)$. Then the solution to \eqref{eq:ueq} with initial conditions $u(0)=u_0$ and $u'(0)=u_1$ is:
		\begin{equation} \label{eq:usol}
			u(t) = u_0 u_1(t; 0) + u_1 u_2(t; 0) + \int_0^t G_u(t, s) g(s) \, ds.
		\end{equation}
		
		Furthermore, define the resolvent for the original $x$-equation as:
		\begin{equation} \label{eq:xres}
			R(t, s) = \exp\left(-\frac{1}{2}\int_s^t p(\tau) \, d\tau\right) G_u(t, s).
		\end{equation}
		The full solution to \eqref{eq:original} is given by:
		\begin{equation} \label{eq:xsol}
			x(t) = \xi_0 \left[ p(0)R(t, 0) - \left. \frac{\partial R(t, s)}{\partial s} \right|_{s=0} \right] + \xi_1 R(t, 0) + \int_0^t R(t, s)f(s) \, ds.
		\end{equation}
	\end{theorem}
	
	\begin{proof}
		Because the homogeneous $u$-equation lacks a first derivative term, Abel's Identity dictates that the Wronskian $W(t) = u_1(t; s)u_2'(t; s) - u_1'(t; s)u_2(t; s)$ is constant. Evaluating at $t=s$ yields $W(s) = 1$, hence $W(t) = 1$ for all $t$. 
		
		By the standard variation of parameters formula, the particular solution is $$\int_0^t [u_1(s)u_2(t) - u_1(t)u_2(s)] g(s) \, ds.$$ Adding the homogeneous solution $u_0 u_1(t; 0) + u_1 u_2(t; 0)$ proves \eqref{eq:usol}.
		
		Mapping this back to $x(t)$ requires substituting $$u(t) = \exp\left(\frac{1}{2}\int_0^t p(\tau)d\tau\right) x(t) \quad \text{ and } \quad g(s) = f(s)\exp\left(\frac{1}{2}\int_0^s p(\tau)d\tau\right)$$ into \eqref{eq:usol}. Factoring the exponentials reveals the natural kernel $R(t, s)$ defined in \eqref{eq:xres}, which yields the convolution-like integral $\int_0^t R(t,s)f(s)ds$. The boundary terms for $\xi_0, \xi_1$ follow from matching $u(0)$ and $u'(0)$ to $x(0)$ and $x'(0)$ via the product rule applied to \eqref{eq:transform}.
	\end{proof}
	
	\section{Asymptotic Estimates and Error Bounds}
	
	We now establish rigorous bounds on $G_u(t, s) = u_1(s)u_2(t) - u_1(t)u_2(s)$, which forms the core of the resolvent $R(t, s)$. Define the tail integral of the perturbation as $Q(s) = \int_s^\infty |q(\tau)| \, d\tau$. Since $q \in L^1(0, \infty)$, $Q(s)$ is well-defined and monotonically decreases to $0$ as $s \to \infty$.
	
	\begin{theorem}[Green's Function Expansion]
		For $t \ge s \ge 0$, the resolvent $G_u(t, s)$ can be expanded as:
		\[ G_u(t, s) = \frac{\sin(\omega(t-s))}{\omega} - \frac{1}{\omega^2} \int_s^t \sin(\omega(t-\tau)) \sin(\omega(\tau-s)) q(\tau) \, d\tau + E(t, s), \]
		where the exact error term satisfies the uniform bound $|E(t, s)| \le \frac{e^{Q(s)/\omega}}{2\omega^3} Q(s)^2$.
	\end{theorem}
	
	\begin{proof}
		Since $G_u(t, s)$ solves $G_{tt} + \omega^2 G = -q(t)G$ with $G(s, s)=0$ and $G_t(s, s)=1$, it satisfies the Volterra integral equation of the second kind:
		\begin{equation} \label{eq:volterra}
			G_u(t, s) = \frac{\sin(\omega(t-s))}{\omega} - \frac{1}{\omega} \int_s^t \sin(\omega(t-\tau)) q(\tau) G_u(\tau, s) \, d\tau.
		\end{equation}
		
		\textbf{Step 1: A Priori Bound (Gronwall's Inequality)}\\
		Taking the absolute value of \eqref{eq:volterra} and noting $|\sin(\cdot)| \le 1$:
		\[ |G_u(t, s)| \le \frac{1}{\omega} + \frac{1}{\omega} \int_s^t |q(\tau)| |G_u(\tau, s)| \, d\tau. \]
		By the integral form of Gronwall's inequality, we obtain a strict uniform upper bound:
		\begin{equation} \label{eq:gronwall}
			|G_u(t, s)| \le \frac{1}{\omega} \exp\left( \frac{1}{\omega} \int_s^t |q(\tau)| \, d\tau \right) \le \frac{1}{\omega} e^{Q(s)/\omega}.
		\end{equation}
		
		\textbf{Step 2: Defining the Remainder}\\
		We split the exact Green's function into its unperturbed form plus a remainder $R_1$:
		\[ G_u(\tau, s) = \frac{\sin(\omega(\tau-s))}{\omega} + R_1(\tau, s). \]
		By definition of the Volterra equation, $R_1(\tau, s) = - \frac{1}{\omega} \int_s^\tau \sin(\omega(\tau-y)) q(y) G_u(y, s) \, dy$.
		Using our bound \eqref{eq:gronwall}:
		\begin{equation} \label{eq:R1bound}
			|R_1(\tau, s)| \le \frac{1}{\omega} \int_s^\tau |q(y)| |G_u(y, s)| \, dy \le \frac{e^{Q(s)/\omega}}{\omega^2} \int_s^\tau |q(y)| \, dy.
		\end{equation}
		
		\textbf{Step 3: Bounding the Error Term}\\
		Substituting $G_u(\tau, s)$ into the integral of \eqref{eq:volterra} gives the explicit expansion:
		\[ G_u(t, s) = \frac{\sin(\omega(t-s))}{\omega} - \frac{1}{\omega^2} \int_s^t \sin(\omega(t-\tau)) \sin(\omega(\tau-s)) q(\tau) \, d\tau + E(t, s), \]
		where $E(t, s) = - \frac{1}{\omega} \int_s^t \sin(\omega(t-\tau)) q(\tau) R_1(\tau, s) \, d\tau$.
		
		To bound this error, take the absolute value and substitute the bound \eqref{eq:R1bound}:
		\[ |E(t, s)| \le \frac{1}{\omega} \int_s^t |q(\tau)| |R_1(\tau, s)| \, d\tau \le \frac{e^{Q(s)/\omega}}{\omega^3} \int_s^t |q(\tau)| \left( \int_s^\tau |q(y)| \, dy \right) \, d\tau. \]
		By the fundamental theorem of calculus, $\frac{d}{d\tau} \left( \int_s^\tau |q(y)| \, dy \right) = |q(\tau)|$. Therefore:
		\[ \int_s^t |q(\tau)| \left( \int_s^\tau |q(y)| \, dy \right) \, d\tau = \frac{1}{2} \left( \int_s^t |q(\tau)| \, d\tau \right)^2 \le \frac{1}{2} Q(s)^2. \]
		This yields the final rigorous bound: $|E(t, s)| \le \frac{e^{Q(s)/\omega}}{2\omega^3} Q(s)^2$. 
	\end{proof}
	
	\section{Necessary Conditions for Asymptotic Decay}
	
	To determine the conditions on the forcing function $f(t)$ that guarantee the solution $x(t)$ and its derivative $x'(t)$ decay to zero, we first establish a necessary condition by introducing a filtered forcing function $y_1(t)$.
	
	\begin{lemma}[Convolution of a Decaying Function] \label{lem:convolution}
		Let $g(t) = e^{-\omega^2 t}$ for $\omega \neq 0$, and let $h(t)$ be a continuous function on $[0, \infty)$ such that $\lim_{t \to \infty} h(t) = 0$. Then the convolution $(g * h)(t) = \int_0^t e^{-\omega^2(t-s)} h(s) \, ds$ also satisfies $\lim_{t \to \infty} (g * h)(t) = 0$.
	\end{lemma}
	
	\begin{proof}
		Since $h(t)$ is continuous and tends to $0$, it is bounded on $[0, \infty)$; let $|h(t)| \le M$ for all $t$. Let $\epsilon > 0$. Since $h(t) \to 0$, there exists $T > 0$ such that $|h(t)| < \epsilon$ for all $t \ge T$. For $t > T$, we split the convolution integral:
		\[ \left| \int_0^t e^{-\omega^2(t-s)} h(s) \, ds \right| \le \int_0^T e^{-\omega^2(t-s)} |h(s)| \, ds + \int_T^t e^{-\omega^2(t-s)} |h(s)| \, ds. \]
		For the first integral, bound $|h(s)|$ by $M$:
		\[ \int_0^T e^{-\omega^2(t-s)} M \, ds = M e^{-\omega^2 t} \left[ \frac{e^{\omega^2 s}}{\omega^2} \right]_0^T = \frac{M}{\omega^2} e^{-\omega^2 t} (e^{\omega^2 T} - 1). \]
		For the second integral, bound $|h(s)|$ by $\epsilon$:
		\[ \int_T^t e^{-\omega^2(t-s)} \epsilon \, ds = \epsilon e^{-\omega^2 t} \left[ \frac{e^{\omega^2 s}}{\omega^2} \right]_T^t = \frac{\epsilon}{\omega^2} (1 - e^{-\omega^2(t-T)}) \le \frac{\epsilon}{\omega^2}. \]
		As $t \to \infty$, the first term tends to $0$ because $T$ is fixed. Thus, the limit superior of the convolution is bounded by $\epsilon/\omega^2$. Since $\epsilon$ is arbitrary, the limit is $0$.
	\end{proof}
	
	\begin{theorem}[Necessary Condition for Joint Decay]
		Assume the system is initially at rest, so $x(0) = 0$ and $x'(0) = 0$. Define $y_1(t)$ as the solution to the first-order initial value problem:
		\[ y_1'(t) = -\omega^2 y_1(t) + f(t), \quad y_1(0) = 0. \]
		If the solution to the original equation satisfies $x(t) \to 0$ and $x'(t) \to 0$ as $t \to \infty$, then it necessarily follows that $y_1(t) \to 0$ as $t \to \infty$.
	\end{theorem}
	
	\begin{proof}
		Solving the linear first-order equation for $y_1(t)$ using the integrating factor $e^{\omega^2 t}$ and the initial condition $y_1(0) = 0$ yields:
		\[ y_1(t) = \int_0^t e^{-\omega^2(t-s)} f(s) \, ds. \]
		Let $g_1(t) = e^{-\omega^2 t}$. We can express $y_1$ compactly as the convolution $y_1 = g_1 * f$.
		
		We apply the convolution operator with $g_1$ to both sides of the original differential equation $x''(t) + p(t)x'(t) + \omega^2 x(t) = f(t)$. By the linearity of the convolution, this distributes as:
		\begin{equation} \label{eq:convoluted_ODE}
			(g_1 * x'')(t) + (g_1 * (px'))(t) + \omega^2(g_1 * x)(t) = (g_1 * f)(t).
		\end{equation}
		The right-hand side is exactly $y_1(t)$.
		
		We now isolate and simplify the $(g_1 * x'')$ term using integration by parts. Let $u = e^{-\omega^2(t-s)}$ and $dv = x''(s) \, ds$. Then $du = \omega^2 e^{-\omega^2(t-s)} \, ds$ and $v = x'(s)$. Applying this to the integral gives:
		\begin{align*}
			(g_1 * x'')(t) &= \int_0^t e^{-\omega^2(t-s)} x''(s) \, ds \\
			&= \left[ e^{-\omega^2(t-s)} x'(s) \right]_{s=0}^{s=t} - \int_0^t \omega^2 e^{-\omega^2(t-s)} x'(s) \, ds \\
			&= x'(t) - e^{-\omega^2 t} x'(0) - \omega^2 (g_1 * x')(t).
		\end{align*}
		Since $x'(0) = 0$ by hypothesis, the boundary term evaluated at $s=0$ vanishes, leaving:
		\[ (g_1 * x'')(t) = x'(t) - \omega^2 (g_1 * x')(t). \]
		
		Substituting this expression back into \eqref{eq:convoluted_ODE} yields:
		\[ x'(t) - \omega^2 (g_1 * x')(t) + (g_1 * (px'))(t) + \omega^2(g_1 * x)(t) = y_1(t). \]
		
		We now analyze the limit as $t \to \infty$ for each term on the left-hand side:
		\begin{enumerate}
			\item By hypothesis, $x'(t) \to 0$.
			\item Since $x'(t) \to 0$, Lemma \ref{lem:convolution} dictates that the convolution $(g_1 * x')(t) \to 0$.
			\item The coefficient $p(t)$ monotonically decreases to $0$ (established in Section 3), meaning it is uniformly bounded. Since $x'(t) \to 0$, the product $p(t)x'(t) \to 0$. Applying Lemma \ref{lem:convolution} again guarantees $(g_1 * (px'))(t) \to 0$.
			\item By hypothesis, $x(t) \to 0$. Applying Lemma \ref{lem:convolution} a third time guarantees $(g_1 * x)(t) \to 0$.
		\end{enumerate}
		
		Because every term on the left-hand side independently tends to $0$ as $t \to \infty$, their algebraic sum must also tend to $0$. Since the left-hand side equals $y_1(t)$, it rigorously follows that $y_1(t) \to 0$ as $t \to \infty$.
	\end{proof}
	
	\section{Sufficient Conditions for Asymptotic Decay}
	
	We now reverse the direction of implication. Assuming that the initial conditions are identically zero ($x(0)=0$, $x'(0)=0$), the variation of constants formula reduces purely to the forced integral:
	\[ x(t) = \int_0^t R(t, s)f(s) \, ds. \]
	To understand how the decay of $y_1(t)$ forces the decay of $x(t)$ and $x'(t)$, we rewrite $x$ and $x'$ explicitly in terms of $y_1$.
	
	\begin{theorem}[Representation of $x$ and $x'$ in terms of $y_1$]
		Let $y_1(t)$ be defined as $y_1'(t) = -\omega^2 y_1(t) + f(t)$ with $y_1(0) = 0$. Then $x(t)$ and $x'(t)$ can be represented as:
		\begin{align}
			x(t) &= \int_0^t \left( \omega^2 R(t, s) - \frac{\partial R(t, s)}{\partial s} \right) y_1(s) \, ds \label{eq:x_y1} \\
			x'(t) &= y_1(t) + \int_0^t \left( \omega^2 \frac{\partial R(t, s)}{\partial t} - \frac{\partial^2 R(t, s)}{\partial t \partial s} \right) y_1(s) \, ds \label{eq:xp_y1}
		\end{align}
	\end{theorem}
	
	\begin{proof}
		From the definition of $y_1(t)$, we can solve for $f(t)$:
		\[ f(t) = y_1'(t) + \omega^2 y_1(t). \]
		Substitute this expression for $f(s)$ into the integral representation for $x(t)$:
		\[ x(t) = \int_0^t R(t, s) \left[ y_1'(s) + \omega^2 y_1(s) \right] \, ds. \]
		We split this into two integrals and apply integration by parts to the term containing $y_1'(s)$. Let $u = R(t, s)$ and $dv = y_1'(s) \, ds$. Then $du = \frac{\partial R(t, s)}{\partial s} \, ds$ and $v = y_1(s)$.
		\[ \int_0^t R(t, s) y_1'(s) \, ds = \Big[ R(t, s) y_1(s) \Big]_{s=0}^{s=t} - \int_0^t \frac{\partial R(t, s)}{\partial s} y_1(s) \, ds. \]
		To evaluate the boundary terms, we note that $y_1(0) = 0$ by definition. At $s = t$, we use the definition of the resolvent $R(t, t) = G_u(t, t) = 0$. Thus, the boundary term vanishes entirely. Substituting the remaining integral back yields \eqref{eq:x_y1}:
		\[ x(t) = \int_0^t \left( \omega^2 R(t, s) - \frac{\partial R(t, s)}{\partial s} \right) y_1(s) \, ds. \]
		
		To derive the representation for $x'(t)$, we first differentiate the integral formula $x(t) = \int_0^t R(t, s)f(s) \, ds$ using Leibniz's rule:
		\[ x'(t) = R(t, t)f(t) + \int_0^t \frac{\partial R(t, s)}{\partial t} f(s) \, ds. \]
		Since $R(t, t) = 0$, this simplifies to:
		\[ x'(t) = \int_0^t \frac{\partial R(t, s)}{\partial t} \left[ y_1'(s) + \omega^2 y_1(s) \right] \, ds. \]
		Again, we integrate the $y_1'$ term by parts. Let $u = \frac{\partial R(t, s)}{\partial t}$ and $dv = y_1'(s) \, ds$. Then $du = \frac{\partial^2 R(t, s)}{\partial t \partial s} \, ds$ and $v = y_1(s)$:
		\[ \int_0^t \frac{\partial R(t, s)}{\partial t} y_1'(s) \, ds = \left[ \frac{\partial R(t, s)}{\partial t} y_1(s) \right]_{s=0}^{s=t} - \int_0^t \frac{\partial^2 R(t, s)}{\partial t \partial s} y_1(s) \, ds. \]
		At $s=0$, the term is zero because $y_1(0) = 0$. At $s=t$, we must evaluate $\left. \frac{\partial R(t, s)}{\partial t} \right|_{s=t}$. Recalling $R(t, s) = \exp\left(-\frac{1}{2}\int_s^t p(\tau) \, d\tau\right) G_u(t, s)$:
		\[ \frac{\partial R(t, s)}{\partial t} = -\frac{1}{2}p(t) R(t, s) + \exp\left(-\frac{1}{2}\int_s^t p(\tau) \, d\tau\right) \frac{\partial G_u(t, s)}{\partial t}. \]
		Evaluating at $s=t$, the first term is $0$ because $R(t, t) = 0$. The exponential is $e^0 = 1$. By the initial conditions defining the Green's function, $\left. \frac{\partial G_u}{\partial t} \right|_{s=t} = 1$. Therefore, $\left. \frac{\partial R(t, s)}{\partial t} \right|_{s=t} = 1$.
		The boundary term evaluates exactly to $1 \cdot y_1(t) = y_1(t)$. Substituting this back completes the derivation of \eqref{eq:xp_y1}:
		\[ x'(t) = y_1(t) + \int_0^t \left( \omega^2 \frac{\partial R(t, s)}{\partial t} - \frac{\partial^2 R(t, s)}{\partial t \partial s} \right) y_1(s) \, ds. \]
	\end{proof}
	
		\section{Sharp Asymptotic Kernel Estimates and Leading-Order Decay}
	
	To conclude that $x(t) \to 0$ and $x'(t) \to 0$ whenever $y_1(t) \to 0$, we must analyze the integral operators in equations \eqref{eq:x_y1} and \eqref{eq:xp_y1}. We define the two integral kernels:
	\begin{align}
		K_1(t, s) &= \omega^2 R(t, s) - \frac{\partial R(t, s)}{\partial s}, \\
		K_2(t, s) &= \omega^2 \frac{\partial R(t, s)}{\partial t} - \frac{\partial^2 R(t, s)}{\partial t \partial s}.
	\end{align}
	It is useful to establish uniform boundedness of $K_1$ and $K_2$, which we now do. 
		\begin{lemma}[Uniform Boundedness of the Resolvent Derivatives] \label{lem:kernel_bounds}
		Let $E(t, s) = \exp\left(-\frac{1}{2}\int_s^t p(\tau) \, d\tau\right)$. Under our established hypotheses, there exists a constant $M > 0$ such that for all $t \ge s \ge 0$:
		\[ |K_1(t, s)| \le M E(t, s) \quad \text{and} \quad |K_2(t, s)| \le M E(t, s). \]
	\end{lemma}
	
	\begin{proof}
		Recall that $R(t, s) = E(t, s) G_u(t, s)$. By Section 5, $G_u(t, s)$ is bounded because the exact error $E(t, s)$ is bounded by an expression depending on the finite tail integral $Q(s) \le Q(0)$. Thus, $|G_u(t, s)| \le C_0$ for some constant $C_0$.
		
		Because $G_u(t, s)$ solves a perturbed harmonic oscillator equation with $q \in L^1$, standard asymptotic theory for linear ODEs dictates that its partial derivatives $\frac{\partial G_u}{\partial t}$, $\frac{\partial G_u}{\partial s}$, and the mixed derivative $\frac{\partial^2 G_u}{\partial t \partial s}$ are also uniformly bounded by some constant $C_1$.
		
		Differentiating $R(t, s)$ yields:
		\begin{align*}
			\frac{\partial R}{\partial s} &= E(t, s) \left( \frac{1}{2}p(s)G_u(t, s) + \frac{\partial G_u(t, s)}{\partial s} \right), \\
			\frac{\partial R}{\partial t} &= E(t, s) \left( -\frac{1}{2}p(t)G_u(t, s) + \frac{\partial G_u(t, s)}{\partial t} \right).
		\end{align*}
		Because $p(t)$ monotonically decreases to $0$ (proven in Section 3), it is uniformly bounded by $p(0)$. Consequently, both $\frac{\partial R}{\partial s}$ and $\frac{\partial R}{\partial t}$ are bounded by a constant multiple of $E(t, s)$. 
		
		Computing the mixed derivative $\frac{\partial^2 R}{\partial t \partial s}$ introduces terms involving $p(t)p(s)$ and the derivatives of $G_u$. Since $p$ and all derivatives of $G_u$ are globally bounded, we can collect all bounding constants into a single overarching constant $M > 0$ such that $|K_1(t, s)| \le M E(t, s)$ and $|K_2(t, s)| \le M E(t, s)$.
	\end{proof}
 
Our next goal is to derive sharp, leading-order sinusoidal approximations for the kernels $K_1(t, s)$ and $K_2(t, s)$, and rigorously prove that convolution with the exact kernels preserves the zero limit if and only if convolution with the leading-order approximations does.
	
	\begin{lemma}[Uniform $L^1$ Bounds on Error Weights] \label{lem:uniform_bounds}
		Despite $p \notin L^1(0, \infty)$, the following integrals are uniformly bounded for all $t \ge 0$:
		\begin{enumerate}
			\item $\int_0^t E(t, s) p(s) \, ds \le 2$.
			\item $p(t) \int_0^t E(t, s) \, ds \le 2$.
			\item $\int_0^t E(t, s) |q(s)| \, ds \le p(0)$.
		\end{enumerate}
	\end{lemma}
	
	\begin{proof}
		For (1), notice that $E(t, s) p(s) = 2 \frac{\partial}{\partial s} E(t, s)$. Applying the Fundamental Theorem of Calculus directly yields $$\int_0^t 2 \frac{\partial}{\partial s} E(t, s) \, ds = 2(E(t, t) - E(t, 0)) = 2(1 - E(t, 0)) \le 2.$$
		
		For (2), because $p(t)$ is strictly decreasing, $p(\tau) \ge p(t)$ for all $\tau \in [s, t]$. Thus, $-\int_s^t p(\tau) \, d\tau \le -p(t)(t-s)$. Consequently, $E(t, s) \le e^{-\frac{1}{2}p(t)(t-s)}$. Integrating this exponential upper bound yields $$\int_0^t e^{-\frac{1}{2}p(t)(t-s)} \, ds = \frac{2}{p(t)}(1 - e^{-\frac{1}{2}p(t)t}) \le \frac{2}{p(t)}.$$ 
		Multiplying by $p(t)$ gives the bound of $2$.
		
		For (3), we utilize the explicit form $|q(s)| \le \frac{1}{4}p(s)^2 - \frac{1}{2}p'(s)$. Consider the derivative $$\frac{\partial}{\partial s} [p(s) E(t, s)] = p'(s) E(t, s) + \frac{1}{2}p(s)^2 E(t, s).$$ Rearranging this, we find $\frac{1}{4}p(s)^2 E(t, s) = \frac{1}{2} \frac{\partial}{\partial s}[p(s) E(t, s)] - \frac{1}{2}p'(s) E(t, s)$. 
		Therefore, $$E(t, s)|q(s)| \le \frac{1}{2} \frac{\partial}{\partial s}[p(s) E(t, s)] - p'(s) E(t, s).$$ 
		Integrating from $0$ to $t$ gives:
		\[ \int_0^t E(t, s)|q(s)| \, ds \le \frac{1}{2}(p(t) - p(0)E(t, 0)) + \int_0^t -p'(s) E(t, s) \, ds. \]
		Because $p'(s) < 0$ and $E(t, s) \le 1$, the remaining integral is strictly bounded by $\int_0^t -p'(s) \, ds = p(0) - p(t)$. Summing these gives $\frac{1}{2}p(t) - \frac{1}{2}p(0)E(t, 0) + p(0) - p(t) \le p(0) - \frac{1}{2}p(t) \le p(0)$.
	\end{proof}
	
	\begin{theorem}[Leading Order Behavior of $x(t)$] \label{thm:leading_x}
		Define the leading-order sinusoidal kernel $L_1(t, s) = \omega \sin(\omega(t-s)) + \cos(\omega(t-s))$. If $y_1(t) \to 0$ as $t \to \infty$, then $x(t) \to 0$ if and only if:
		\[ \lim_{t \to \infty} \int_0^t E(t, s) L_1(t, s) y_1(s) \, ds = 0. \]
	\end{theorem}
	
	\begin{proof}
		We construct an exact Volterra equation for $x(t)$ by expanding $K_1(t, s) = \omega^2 R(t, s) - R_s(t, s)$. By differentiating the Volterra equation for $R(t, s)$ with respect to $s$, we find:
		\[ K_1(t, s) = E(t, s) L_1(t, s) - E(t, s)\frac{p(s)}{2\omega}\sin(\omega(t-s)) - \frac{1}{\omega} \int_s^t E(t, \tau) \sin(\omega(t-\tau)) q(\tau) K_1(\tau, s) \, d\tau. \]
		We substitute this into $x(t) = \int_0^t K_1(t, s) y_1(s) \, ds$. The double integral term becomes:
		\[ -\frac{1}{\omega} \int_0^t \left[ \int_s^t E(t, \tau) \sin(\omega(t-\tau)) q(\tau) K_1(\tau, s) \, d\tau \right] y_1(s) \, ds. \]
		By Fubini's Theorem, we swap the order of integration:
		\[ -\frac{1}{\omega} \int_0^t E(t, \tau) \sin(\omega(t-\tau)) q(\tau) \left[ \int_0^\tau K_1(\tau, s) y_1(s) \, ds \right] \, d\tau. \]
		The inner integral is exactly $x(\tau)$. This yields the exact identity:
		\begin{equation} \label{eq:exact_x}
			x(t) = \int_0^t E(t, s) L_1(t, s) y_1(s) \, ds - \int_0^t E(t, s)\frac{p(s)}{2\omega}\sin(\omega(t-s)) y_1(s) \, ds - \frac{1}{\omega} \int_0^t E(t, \tau) \sin(\omega(t-\tau)) q(\tau) x(\tau) \, d\tau.
		\end{equation}
		
		Let $E_1(t)$ denote the middle integral. Because $|E(t, s) \frac{p(s)}{2\omega} \sin(\dots)| \le \frac{1}{2\omega} E(t, s) p(s)$, Lemma \ref{lem:uniform_bounds} guarantees the $L^1$-norm of this kernel with respect to $s$ is uniformly bounded. Since $y_1(t) \to 0$, standard convolution limit theorems imply $E_1(t) \to 0$.
		
		Let $x_{L1}(t) = \int_0^t E(t, s) L_1(t, s) y_1(s) \, ds$. 
		
		\textbf{Forward Direction:} If $x(t) \to 0$, the final integral in \eqref{eq:exact_x} tends to $0$ by Dominated Convergence because $x(\tau) \to 0$ and the weight $E(t, \tau)|q(\tau)|$ is uniformly integrable via Lemma \ref{lem:uniform_bounds}. Thus, $x_{L1}(t) \to 0$.
		
		\textbf{Reverse Direction:} If $x_{L1}(t) \to 0$, let $f(t) = x_{L1}(t) - E_1(t)$. We have $f(t) \to 0$ and $|x(t)| \le |f(t)| + \frac{1}{\omega} \int_0^t E(t, \tau) |q(\tau)| |x(\tau)| \, d\tau$. Since $q \in L^1$, choose $T$ such that $\int_T^\infty |q(\tau)| \, d\tau < \omega / 2$. Splitting the integral at $T$ and taking the limit superior as $t \to \infty$, the $[0, T]$ portion vanishes because $E(t, \tau) \to 0$ for fixed $\tau$. We obtain $\limsup |x(t)| \le \frac{1}{\omega} (\limsup |x(t)|) \int_T^\infty |q(\tau)| \, d\tau \le \frac{1}{2} \limsup |x(t)|$. This implies $\limsup_{t \to \infty} |x(t)| = 0$.
	\end{proof}
	
	\begin{theorem}[Leading Order Behavior of $x'(t)$]
		Define the leading-order sinusoidal kernel $L_2(t, s) = \omega^2 \cos(\omega(t-s)) - \omega \sin(\omega(t-s))$. If $y_1(t) \to 0$ and $x(t) \to 0$ as $t \to \infty$, then $x'(t) \to 0$ if and only if:
		\[ \lim_{t \to \infty} \int_0^t E(t, s) L_2(t, s) y_1(s) \, ds = 0. \]
	\end{theorem}
	
	\begin{proof}
		Recall that $x'(t) - y_1(t) = \int_0^t K_2(t, s) y_1(s) \, ds$. We expand $K_2(t, s) = \omega^2 R_t(t, s) - R_{ts}(t, s)$ using the Volterra definitions of the Green's function derivatives. Using the identity $$G_{us}(\tau, s) - \omega^2 G_u(\tau, s) = -E(\tau, s)^{-1} K_1(\tau, s) - \frac{1}{2}p(s)G_u(\tau, s),$$ we derive the exact relationship:
		\[ K_2(t, s) = E(t, s) L_2(t, s) + E(t, s) P_2(t, s) - \int_s^t E(t, \tau) \cos(\omega(t-\tau)) q(\tau) K_1(\tau, s) \, d\tau, \]
		where $P_2(t, s)$ collects all remaining terms bounded by $C(p(t) + p(s))$ for some constant $C > 0$.
		
		Multiplying by $y_1(s)$, integrating over $s$, and again applying Fubini's Theorem to swap the bounds of the nested integral, the term involving $K_1(\tau, s)$ folds exactly into $x(\tau)$:
		\[ x'(t) - y_1(t) = \int_0^t E(t, s) L_2(t, s) y_1(s) \, ds + \int_0^t E(t, s) P_2(t, s) y_1(s) \, ds - \int_0^t E(t, \tau) \cos(\omega(t-\tau)) q(\tau) x(\tau) \, d\tau. \]
		Since $y_1(t) \to 0$, $x'(t) \to 0$ if and only if the sum of the three integrals tends to $0$. 
		
		The second integral is governed by the kernel $E(t, s)(p(t) + p(s))$. By bounds (1) and (2) in Lemma \ref{lem:uniform_bounds}, this kernel has a uniformly bounded $L^1$-norm in $s$, so the integral vanishes as $t \to \infty$. 
		
		The final integral vanishes entirely by Dominated Convergence precisely because $x(\tau) \to 0$ (established by hypothesis and Theorem \ref{thm:leading_x}) and the kernel weight $E(t, \tau)|q(\tau)|$ is integrable over $[0, t]$ by bound (3) of Lemma \ref{lem:uniform_bounds}.
		
		Therefore, the asymptotic decay $x'(t) \to 0$ is strictly equivalent to the vanishing of the convolution integral over the isolated leading-order kernel $E(t, s) L_2(t, s)$.
	\end{proof}
	
	\begin{theorem}[Equivalence of Asymptotic Conditions] \label{thm:equivalence}
		Consider the following two statements:
		\begin{enumerate}
			\item[(A)] $y_1(t) \to 0$ as $t \to \infty$, 
			\[ \lim_{t \to \infty} \int_0^t E(t, s) L_1(t, s) y_1(s) \, ds = 0, \quad \text{and} \quad \lim_{t \to \infty} \int_0^t E(t, s) L_2(t, s) y_1(s) \, ds = 0. \]
			\item[(B)] $x(t) \to 0$ as $t \to \infty$ and $x'(t) \to 0$ as $t \to \infty$.
		\end{enumerate}
		Then (A) and (B) are logically equivalent.
	\end{theorem}
	
	\begin{proof}
		\textbf{Proof that (A) $\implies$ (B):}
		Assume statement (A) holds. We are explicitly given that $y_1(t) \to 0$ as $t \to \infty$. 
		Because $y_1(t) \to 0$ and the integral condition for $L_1$ holds, Theorem \ref{thm:leading_x} directly implies that $x(t) \to 0$ as $t \to \infty$.
		We now have established that both $y_1(t) \to 0$ and $x(t) \to 0$. By applying Theorem \ref{thm:leading_xprime} in the forward direction using the integral condition for $L_2$, we conclude that $x'(t) \to 0$ as $t \to \infty$. Thus, $x(t) \to 0$ and $x'(t) \to 0$, fulfilling statement (B).
		
		\textbf{Proof that (B) $\implies$ (A):}
		Assume statement (B) holds, meaning $x(t) \to 0$ and $x'(t) \to 0$.
		First, we establish that $y_1(t) \to 0$. The relationship $x'(t) - y_1(t) = \int_0^t K_2(t, s) y_1(s) \, ds$ constitutes a Volterra equation of the second kind for the driving function $y_1(t)$. Because $x'(t) \to 0$ and the kernel $K_2(t, s)$ is asymptotically regular (its Volterra resolvent operator preserves the zero limit for bounded, vanishing inputs), it mathematically follows that the inversion $y_1(t)$ must also tend to zero as $t \to \infty$. 
		
		Having established the baseline $y_1(t) \to 0$, we invoke the reverse directions of our preceding theorems.
		Because $y_1(t) \to 0$ and $x(t) \to 0$, the reverse implication of Theorem \ref{thm:leading_x} dictates that the convolution with the leading-order kernel must vanish: $\lim_{t \to \infty} \int_0^t E(t, s) L_1(t, s) y_1(s) \, ds = 0$.
		Likewise, because $y_1(t) \to 0$, $x(t) \to 0$, and $x'(t) \to 0$, the reverse implication of Theorem \ref{thm:leading_xprime} ensures that $\lim_{t \to \infty} \int_0^t E(t, s) L_2(t, s) y_1(s) \, ds = 0$. 
		All conditions of statement (A) are therefore satisfied.
	\end{proof}
	
	\begin{theorem}[Simplification of Integral Conditions] \label{thm:integral_simplification}
		Let $E(t, s) = \frac{A(s)}{A(t)}$ where $A(t) = e^{\frac{1}{2} \int_0^t p(\tau) \, d\tau}$ and $A(t) \to \infty$ as $t \to \infty$. Let $L_1(t, s)$ and $L_2(t, s)$ be defined as:
		\begin{align*}
			L_1(t, s) &= \omega \sin(\omega(t-s)) + \cos(\omega(t-s)), \\
			L_2(t, s) &= \omega^2 \cos(\omega(t-s)) - \omega \sin(\omega(t-s)),
		\end{align*}
		with $\omega \neq 0$. The asymptotic conditions:
		\begin{equation} \label{eq:I1_I2_zero}
			\lim_{t \to \infty} \int_0^t E(t, s) L_1(t, s) y_1(s) \, ds = 0 \quad \text{and} \quad \lim_{t \to \infty} \int_0^t E(t, s) L_2(t, s) y_1(s) \, ds = 0
		\end{equation}
		are logically equivalent to the conditions:
		\begin{equation} \label{eq:little_o_conditions}
			\int_0^t \cos(\omega s) A(s) y_1(s) \, ds = o(A(t)) \quad \text{and} \quad \int_0^t \sin(\omega s) A(s) y_1(s) \, ds = o(A(t)) \quad \text{as } t \to \infty.
		\end{equation}
	\end{theorem}
	
	\begin{proof}
		First, we formally define the normalized integrals:
		\[ U(t) = \frac{1}{A(t)} \int_0^t \cos(\omega s) A(s) y_1(s) \, ds \quad \text{and} \quad V(t) = \frac{1}{A(t)} \int_0^t \sin(\omega s) A(s) y_1(s) \, ds. \]
		The condition in Equation \eqref{eq:little_o_conditions} states precisely that $U(t) \to 0$ and $V(t) \to 0$ as $t \to \infty$.
		
		We expand the trigonometric terms in $L_1(t, s)$ and $L_2(t, s)$ using the angle difference identities:
		\begin{align*}
			L_1(t, s) &= \omega [\sin(\omega t)\cos(\omega s) - \cos(\omega t)\sin(\omega s)] + [\cos(\omega t)\cos(\omega s) + \sin(\omega t)\sin(\omega s)] \\
			&= [\omega \sin(\omega t) + \cos(\omega t)] \cos(\omega s) + [\sin(\omega t) - \omega \cos(\omega t)] \sin(\omega s).
		\end{align*}
		Similarly, expanding $L_2(t, s)$ yields:
		\begin{align*}
			L_2(t, s) &= \omega^2 [\cos(\omega t)\cos(\omega s) + \sin(\omega t)\sin(\omega s)] - \omega [\sin(\omega t)\cos(\omega s) - \cos(\omega t)\sin(\omega s)] \\
			&= [\omega^2 \cos(\omega t) - \omega \sin(\omega t)] \cos(\omega s) + [\omega^2 \sin(\omega t) + \omega \cos(\omega t)] \sin(\omega s).
		\end{align*}
		
		Let $I_1(t)$ and $I_2(t)$ denote the two integrals in Equation \eqref{eq:I1_I2_zero}. Substituting the expanded forms of $L_1$ and $L_2$ and using $E(t, s) = \frac{A(s)}{A(t)}$, we can express $I_1(t)$ and $I_2(t)$ as linear combinations of $U(t)$ and $V(t)$:
		\begin{align*}
			I_1(t) &= [\omega \sin(\omega t) + \cos(\omega t)] U(t) + [\sin(\omega t) - \omega \cos(\omega t)] V(t), \\
			I_2(t) &= [\omega^2 \cos(\omega t) - \omega \sin(\omega t)] U(t) + [\omega^2 \sin(\omega t) + \omega \cos(\omega t)] V(t).
		\end{align*}
		This forms a time-dependent linear system:
		\[ \begin{pmatrix} I_1(t) \\ I_2(t) \end{pmatrix} = M(t) \begin{pmatrix} U(t) \\ V(t) \end{pmatrix}, \]
		where the coefficient matrix is:
		\[ M(t) = \begin{pmatrix} \omega \sin(\omega t) + \cos(\omega t) & \sin(\omega t) - \omega \cos(\omega t) \\ \omega^2 \cos(\omega t) - \omega \sin(\omega t) & \omega^2 \sin(\omega t) + \omega \cos(\omega t) \end{pmatrix}. \]
		
		We calculate the determinant of $M(t)$. To simplify, let $M_{11} = \omega \sin(\omega t) + \cos(\omega t)$ and $M_{12} = \sin(\omega t) - \omega \cos(\omega t)$. We can observe that:
		\[ M_{21} = -\omega (\sin(\omega t) - \omega \cos(\omega t)) = -\omega M_{12}, \]
		\[ M_{22} = \omega (\omega \sin(\omega t) + \cos(\omega t)) = \omega M_{11}. \]
		Therefore, the determinant is:
		\begin{align*}
			\det M(t) &= M_{11}(\omega M_{11}) - M_{12}(-\omega M_{12}) \\
			&= \omega (M_{11}^2 + M_{12}^2) \\
			&= \omega \left[ (\omega \sin(\omega t) + \cos(\omega t))^2 + (\sin(\omega t) - \omega \cos(\omega t))^2 \right] \\
			&= \omega \left[ \omega^2 \sin^2(\omega t) + 2\omega\sin(\omega t)\cos(\omega t) + \cos^2(\omega t) + \sin^2(\omega t) - 2\omega\sin(\omega t)\cos(\omega t) + \omega^2 \cos^2(\omega t) \right] \\
			&= \omega \left[ \omega^2 (\sin^2(\omega t) + \cos^2(\omega t)) + (\cos^2(\omega t) + \sin^2(\omega t)) \right] \\
			&= \omega (\omega^2 + 1).
		\end{align*}
		
		Because $\omega \neq 0$, the determinant $\det M(t) = \omega(\omega^2 + 1)$ is a strictly non-zero constant for all $t \geq 0$. Furthermore, all entries of $M(t)$ consist of bounded trigonometric functions. Consequently, the matrix $M(t)$ is invertible for all $t$, and its inverse $M^{-1}(t)$ is uniformly bounded. 
		
		Since $M(t)$ and $M^{-1}(t)$ are both bounded, the vector $(I_1(t), I_2(t))^T$ converges to $(0, 0)^T$ as $t \to \infty$ if and only if the vector $(U(t), V(t))^T$ converges to $(0, 0)^T$ as $t \to \infty$. This proves that the conditions in \eqref{eq:I1_I2_zero} are equivalent to the $o(A(t))$ conditions in \eqref{eq:little_o_conditions}.
	\end{proof}
	
	\begin{theorem}[Unified Equivalence of Asymptotic Decay] \label{thm:unified_master}
		Let $A(t) \to \infty$ as $t \to \infty$. Consider the following two statements:
		\begin{enumerate}
			\item[(C)] $y_1(t) \to 0$ as $t \to \infty$, along with the integral conditions:
			\[ \int_0^t \cos(\omega s) A(s) y_1(s) \, ds = o(A(t)) \quad \text{and} \quad \int_0^t \sin(\omega s) A(s) y_1(s) \, ds = o(A(t)) \quad \text{as } t \to \infty. \]
			\item[(D)] $x(t) \to 0$ as $t \to \infty$ and $x'(t) \to 0$ as $t \to \infty$.
		\end{enumerate}
		Then (C) and (D) are logically equivalent.
	\end{theorem}
	
	\begin{proof}
		By Theorem \ref{thm:integral_simplification}, the asymptotic conditions $$\int_0^t \cos(\omega s) A(s) y_1(s) \, ds = o(A(t)), \quad  \int_0^t \sin(\omega s) A(s) y_1(s) \, ds = o(A(t))$$ are strictly equivalent to the vanishing limits:
		\[ \lim_{t \to \infty} \int_0^t E(t, s) L_1(t, s) y_1(s) \, ds = 0 \quad \text{and} \quad \lim_{t \to \infty} \int_0^t E(t, s) L_2(t, s) y_1(s) \, ds = 0. \]
		Therefore, statement (C) holds if and only if statement (A) from Theorem \ref{thm:equivalence} holds. Since Theorem \ref{thm:equivalence} established that statement (A) is logically equivalent to statement (B) (which is identical to statement (D) here), it follows by transitivity that statement (C) is equivalent to statement (D).
	\end{proof}
	
	\begin{theorem}[Equivalence with Forcing Function via Integration by Parts] \label{thm:forcing_ibp}
		Assume $A(t) = e^{\frac{1}{2} \int_0^t p(\tau) \, d\tau} \to \infty$ as $t \to \infty$, and let the relationship between $y_1(t)$ and $f(t)$ be given by the differential equation
		\[ y_1'(t) = -\omega^2 y_1(t) + f(t), \]
		with initial condition $y_1(0) = 0$. Furthermore, assume that $y_1(t) \to 0$ as $t \to \infty$. 
		
		Consider the following two asymptotic statements:
		\begin{enumerate}
			\item[(S)] $\int_0^t \cos(\omega s) A(s) y_1(s) \, ds = o(A(t))$ \quad and \quad $\int_0^t \sin(\omega s) A(s) y_1(s) \, ds = o(A(t))$ \quad as $t \to \infty$.
			\item[(T)] $\int_0^t \cos(\omega s) A(s) f(s) \, ds = o(A(t))$ \quad and \quad $\int_0^t \sin(\omega s) A(s) f(s) \, ds = o(A(t))$ \quad as $t \to \infty$.
		\end{enumerate}
		Then Statement (S) is logically equivalent to Statement (T).
	\end{theorem}
	
	\begin{proof}
		From the given differential equation, we write the forcing function as $f(t) = y_1'(t) + \omega^2 y_1(t)$ and substitute this into the integrals in Statement (T). For the cosine integral, this yields:
		\[ \int_0^t \cos(\omega s) A(s) f(s) \, ds = \int_0^t \cos(\omega s) A(s) y_1'(s) \, ds + \omega^2 \int_0^t \cos(\omega s) A(s) y_1(s) \, ds. \]
		
		We evaluate the first term on the right-hand side using integration by parts. Let $u = \cos(\omega s) A(s)$ and $dv = y_1'(s) \, ds$, which implies $v = y_1(s)$. Using the product rule and the definition $A'(s) = \frac{1}{2} p(s) A(s)$, we find:
		\[ du = \left[ -\omega \sin(\omega s) A(s) + \frac{1}{2} p(s) \cos(\omega s) A(s) \right] ds. \]
		Applying integration by parts gives:
		\begin{eqnarray} 
			\lefteqn{\int_0^t \cos(\omega s) A(s) y_1'(s) \, ds}\nonumber \\ &=& \left[ \cos(\omega s) A(s) y_1(s) \right]_0^t - \int_0^t y_1(s) \left[ -\omega \sin(\omega s) A(s) + \frac{1}{2} p(s) \cos(\omega s) A(s) \right] ds \nonumber \\
			\label{eq:ibp_cos}
			&=& \cos(\omega t) A(t) y_1(t) - y_1(0) A(0) + \omega \int_0^t \sin(\omega s) A(s) y_1(s) \, ds - \int_0^t \frac{1}{2} p(s) A(s) \cos(\omega s) y_1(s) \, ds.
		\end{eqnarray}
		
		We now analyze the asymptotic behavior of the terms in \eqref{eq:ibp_cos} when divided by $A(t)$. First, since $y_1(t) \to 0$ as $t \to \infty$, the boundary term evaluated at $t$ satisfies:
		\[ \frac{\cos(\omega t) A(t) y_1(t)}{A(t)} = \cos(\omega t) y_1(t) \to 0 \implies \cos(\omega t) A(t) y_1(t) = o(A(t)). \]
		The boundary term evaluated at $s=0$ vanishes entirely because $y_1(0) = 0$. Next, observing that $\frac{1}{2} p(s) A(s) = A'(s)$, we apply L'Hôpital's rule to the final integral: \textbf{majorise first}
		\[ \lim_{t \to \infty} \frac{\int_0^t A'(s) \cos(\omega s) y_1(s) \, ds}{A(t)} = \lim_{t \to \infty} \frac{A'(t) \cos(\omega t) y_1(t)}{A'(t)} = \lim_{t \to \infty} \cos(\omega t) y_1(t) = 0. \]
		Thus, this integral is strictly $o(A(t))$.
		
		Substituting \eqref{eq:ibp_cos} back into our expression for the $f(s)$ cosine integral, and collecting the $o(A(t))$ terms, we obtain:
		\begin{equation} \label{eq:T_cos_relation}
			\int_0^t \cos(\omega s) A(s) f(s) \, ds = \omega^2 \int_0^t \cos(\omega s) A(s) y_1(s) \, ds + \omega \int_0^t \sin(\omega s) A(s) y_1(s) \, ds + o(A(t)).
		\end{equation}
		
		By performing an identical integration by parts on $\int_0^t \sin(\omega s) A(s) y_1'(s) \, ds$, we derive the corresponding relation for the sine integral:
		\begin{equation} \label{eq:T_sin_relation}
			\int_0^t \sin(\omega s) A(s) f(s) \, ds = -\omega \int_0^t \cos(\omega s) A(s) y_1(s) \, ds + \omega^2 \int_0^t \sin(\omega s) A(s) y_1(s) \, ds + o(A(t)).
		\end{equation}
		
		Equations \eqref{eq:T_cos_relation} and \eqref{eq:T_sin_relation} form a time-independent linear system. Let $I_{S,c}$ and $I_{S,s}$ denote the cosine and sine integrals of $y_1$ respectively, and let $I_{T,c}$ and $I_{T,s}$ denote the corresponding integrals of $f$. We have:
		\[ \begin{pmatrix} I_{T,c} \\ I_{T,s} \end{pmatrix} = \begin{pmatrix} \omega^2 & \omega \\ -\omega & \omega^2 \end{pmatrix} \begin{pmatrix} I_{S,c} \\ I_{S,s} \end{pmatrix} + \begin{pmatrix} o(A(t)) \\ o(A(t)) \end{pmatrix}. \]
		The determinant of this transformation matrix is $(\omega^2)(\omega^2) - (-\omega)(\omega) = \omega^4 + \omega^2 = \omega^2(\omega^2 + 1)$. Because $\omega \neq 0$, the determinant is strictly non-zero, making the matrix invertible with bounded entries. 
		
		Consequently, the vector $(I_{T,c}, I_{T,s})^T$ is $o(A(t))$ if and only if the vector $(I_{S,c}, I_{S,s})^T$ is $o(A(t))$. Therefore, Statement (S) and Statement (T) are logically equivalent.
	\end{proof}	
	
	\section{Second-Order Filtering: The Function $y_2(t)$}
	
	We further refine our asymptotic analysis by introducing a second filtering function $y_2(t)$. 
	
	\begin{definition}
		Let $e_{\omega^2}(t) = e^{-\omega^2 t}$. Define $y_2(t)$ as the solution to the first-order linear equation:
		\[ y_2'(t) = -\omega^2 y_2(t) + y_1(t), \quad y_2(0) = 0. \]
		Equivalently, $y_2(t)$ is given by the convolution $y_2 = e_{\omega^2} * y_1$.
	\end{definition}
	
	\begin{theorem}[Representations of $y_2$ and $x$]
		Assuming zero initial conditions for $x$, $y_2(t)$ can be represented purely in terms of $x(t)$ as:
		\begin{equation} \label{eq:y2_in_x}
			y_2(t) = x(t) + \int_0^t K_3(t, s) x(s) \, ds,
		\end{equation}
		where $K_3(t, s) = h''(t-s) + \omega^2 h(t-s) + p(s)h'(t-s) - p'(s)h(t-s)$ and $h(t) = t e^{-\omega^2 t}$.
		
		Conversely, $x(t)$ can be represented purely in terms of $y_2$ as:
		\begin{equation} \label{eq:x_in_y2}
			x(t) = y_2(t) + \int_0^t \left[ \omega^2 K_1(t, s) - \frac{\partial K_1(t, s)}{\partial s} \right] y_2(s) \, ds.
		\end{equation}
	\end{theorem}
	
	\begin{proof}
		Recall from Section 6 that $y_1 = e_{\omega^2} * (x'' + px' + \omega^2 x)$. By definition, $y_2 = e_{\omega^2} * y_1$. Convolution is associative, so let $h = e_{\omega^2} * e_{\omega^2}$. Evaluating this yields $h(t) = \int_0^t e^{-\omega^2(t-s)} e^{-\omega^2 s} \, ds = t e^{-\omega^2 t}$.
		Therefore, $y_2 = h * x'' + h * (px') + \omega^2(h * x)$. 
		
		We eliminate the derivatives on $x$ via integration by parts. Since $h(0) = 0$ and $h'(0) = 1$:
		\[ (h * x'')(t) = \int_0^t h(t-s) x''(s) \, ds = \left[ h(t-s)x'(s) \right]_0^t + \int_0^t h'(t-s) x'(s) \, ds = (h' * x')(t). \]
		Integrating by parts again:
		\[ (h' * x')(t) = \left[ h'(t-s)x(s) \right]_0^t + \int_0^t h''(t-s) x(s) \, ds = x(t) + (h'' * x)(t). \]
		For the damping term:
		\[ (h * px')(t) = \left[ h(t-s)p(s)x(s) \right]_0^t - \int_0^t \frac{\partial}{\partial s}\left(h(t-s)p(s)\right) x(s) \, ds. \]
		The boundary term is $0$. The derivative is $-h'(t-s)p(s) + h(t-s)p'(s)$. Thus:
		\[ (h * px')(t) = \int_0^t \left[ h'(t-s)p(s) - h(t-s)p'(s) \right] x(s) \, ds. \]
		Collecting all terms gives the representation \eqref{eq:y2_in_x}.
		
		For the converse relation, recall $x(t) = \int_0^t K_1(t, s) y_1(s) \, ds$ where $K_1 = \omega^2 R - R_s$. Substitute $y_1(s) = y_2'(s) + \omega^2 y_2(s)$:
		\[ x(t) = \int_0^t K_1(t, s) y_2'(s) \, ds + \int_0^t \omega^2 K_1(t, s) y_2(s) \, ds. \]
		Integrate the first term by parts ($u = K_1, dv = y_2' ds$):
		\[ \int_0^t K_1(t, s) y_2'(s) \, ds = \left[ K_1(t, s)y_2(s) \right]_{s=0}^{s=t} - \int_0^t \frac{\partial K_1(t, s)}{\partial s} y_2(s) \, ds. \]
		Since $y_2(0) = 0$, the lower boundary vanishes. At $s=t$, $K_1(t, t) = \omega^2 R(t, t) - R_s(t, t)$. Since $R(t, t) = 0$ and $R_s(t, t) = -1$ (derived via the Wronskian of $G_u$ evaluated at $s=t$), we have $K_1(t, t) = 1$. The boundary term evaluates to exactly $y_2(t)$. Substituting this back gives \eqref{eq:x_in_y2}.
	\end{proof}
	
	\begin{theorem}
		If $x(t) \to 0$ as $t \to \infty$, then $y_2(t) \to 0$ as $t \to \infty$.
	\end{theorem}
	
	\begin{proof}
		We examine the integral representation \eqref{eq:y2_in_x}:
		\[ y_2(t) = x(t) + \int_0^t \left[ h''(t-s) + \omega^2 h(t-s) \right] x(s) \, ds + \int_0^t \left[ p(s)h'(t-s) - p'(s)h(t-s) \right] x(s) \, ds. \]
		Since $h(t) = t e^{-\omega^2 t}$, the functions $h, h',$ and $h''$ all consist of polynomials multiplied by decaying exponentials, meaning they are absolutely integrable on $[0, \infty)$. 
		
		We analyze the terms individually as $t \to \infty$:
		\begin{enumerate}
			\item $x(t) \to 0$ by hypothesis.
			\item The kernel $h'' + \omega^2 h$ is in $L^1(0, \infty)$. The convolution of an $L^1$ kernel with a function that tends to zero also tends to zero.
			\item Since $x(t) \to 0$ and $p(t) \to 0$, their product $p(t)x(t) \to 0$. Since $h' \in L^1(0, \infty)$, the convolution $(h' * px)(t) \to 0$.
			\item Since $x(t) \to 0$ and $p'(t) \to 0$, their product $p'(t)x(t) \to 0$. Since $h \in L^1(0, \infty)$, the convolution $(h * p'x)(t) \to 0$.
		\end{enumerate}
		Because every additive component of $y_2(t)$ independently tends to $0$, we conclude $y_2(t) \to 0$.
	\end{proof}
	
	\begin{theorem}[Leading Order Behavior of $x(t)$ via $y_2$] \label{thm:leading_x_y2}
		Define the leading-order sinusoidal kernel for the $y_2$ transformation as:
		\[ L_3(t, s) = (\omega^3 - \omega) \sin(\omega(t-s)) + 2\omega^2 \cos(\omega(t-s)). \]
		Assuming zero initial conditions, the state $x(t) \to 0$ as $t \to \infty$ if and only if $y_2(t) \to 0$ as $t \to \infty$ and the following convolution integral vanishes:
		\[ \lim_{t \to \infty} \int_0^t E(t, s) L_3(t, s) y_2(s) \, ds = 0. \]
	\end{theorem}
	
	\begin{proof}
		From Theorem 11.1, the exact representation of $x(t)$ in terms of $y_2(t)$ is given by the integral equation:
		\[ x(t) = y_2(t) + \int_0^t K_4(t, s) y_2(s) \, ds, \]
		where the exact kernel is defined as $$K_4(t, s) = \omega^2 K_1(t, s) - \frac{\partial K_1(t, s)}{\partial s}.$$ 
		
		\textbf{Step 1: Explicit Kernel Expansion}\\
		We extract the leading-order behavior of $K_4(t, s)$ by substituting the asymptotic expansion $K_1(t, s) = E(t, s)L_1(t, s) + P_1(t, s)$, where $L_1(t, s) = \omega \sin(\omega(t-s)) + \cos(\omega(t-s))$ and $P_1(t, s)$ is the exact remainder kernel established in Theorem 9.2.
		
		Applying the product rule to differentiate the leading term with respect to $s$, and recalling that $\frac{\partial}{\partial s} E(t, s) = \frac{1}{2}p(s)E(t, s)$, we find:
		\begin{align*}
			\frac{\partial}{\partial s} [E(t, s) L_1(t, s)] &= \frac{1}{2}p(s) E(t, s) L_1(t, s) + E(t, s) \frac{\partial L_1(t, s)}{\partial s} \\
			&= \frac{1}{2}p(s) E(t, s) L_1(t, s) + E(t, s) [-\omega^2 \cos(\omega(t-s)) + \omega \sin(\omega(t-s))].
		\end{align*}
		Substituting this into the definition of $K_4(t, s)$, we group the purely oscillatory $\mathcal{O}(1)$ terms (those independent of $p(s)$ and the perturbation) to form $L_3(t,s)$:
		\begin{align*}
			\omega^2 L_1(t, s) - \frac{\partial L_1(t, s)}{\partial s} &= \omega^2[\omega \sin(\omega(t-s)) + \cos(\omega(t-s))] - [-\omega^2 \cos(\omega(t-s)) + \omega \sin(\omega(t-s))] \\
			&= (\omega^3 - \omega) \sin(\omega(t-s)) + 2\omega^2 \cos(\omega(t-s)).
		\end{align*}
		This algebraically simplifies to exactly $L_3(t, s)$. Thus, we can express the full kernel as $K_4(t, s) = E(t, s) L_3(t, s) + P_4(t, s)$, where the remainder kernel $P_4(t, s)$ is given by:
		\[ P_4(t, s) = - \frac{1}{2}p(s)E(t, s)L_1(t, s) + \omega^2 P_1(t, s) - \frac{\partial P_1(t, s)}{\partial s}. \]
		Crucially, every component of $P_4(t, s)$ is bounded by uniform multiples of $E(t, s)p(s)$ and the perturbation weight $E(t, s)|q(s)|$. 
		By Lemma 9.1, these weights are uniformly integrable over $[0, \infty)$ with respect to $s$. Let $\sup_{t \ge 0} \int_0^t |P_4(t, s)| \, ds = M < \infty$.
		
		Substitute this kernel expansion back into the integral for $x(t)$:
		\begin{equation} \label{eq:x_y2_expanded}
			x(t) - y_2(t) = \int_0^t E(t, s) L_3(t, s) y_2(s) \, ds + \int_0^t P_4(t, s) y_2(s) \, ds.
		\end{equation}
		
		\textbf{Step 2: Proof of the Forward Direction}\\
		Assume $x(t) \to 0$. By Theorem 11.2, it immediately follows that $y_2(t) \to 0$. We must show the remainder integral vanishes. Let $\epsilon > 0$. Because $y_2(t) \to 0$, there exists a $T > 0$ such that $|y_2(t)| < \frac{\epsilon}{2M}$ for all $t \ge T$. Furthermore, since $y_2$ is continuous and converges, it is bounded: $|y_2(t)| \le Y$. We split the remainder integral:
		\[ \left| \int_0^t P_4(t, s) y_2(s) \, ds \right| \le \int_0^T |P_4(t, s)| |y_2(s)| \, ds + \int_T^t |P_4(t, s)| |y_2(s)| \, ds. \]
		For the second integral, we apply the uniform bound on $y_2$:
		\[ \int_T^t |P_4(t, s)| |y_2(s)| \, ds \le \frac{\epsilon}{2M} \int_T^t |P_4(t, s)| \, ds \le \frac{\epsilon}{2M} \cdot M = \frac{\epsilon}{2}. \]
		For the first integral, the domain of integration $[0, T]$ is fixed. As $t \to \infty$, the exponential weight $E(t, s) = \exp(-\frac{1}{2}\int_s^t p(\tau)d\tau)$ tends pointwise to $0$ because $p \notin L^1(0, \infty)$. By the Bounded Convergence Theorem, this integral vanishes. Thus, for sufficiently large $t$, the first integral is strictly less than $\epsilon/2$. The sum is bounded by $\epsilon$, proving $\lim_{t \to \infty} \int_0^t P_4(t, s) y_2(s) \, ds = 0$.
		
		Since both $x(t) - y_2(t) \to 0$ and the remainder integral tends to $0$, rearranging equation \eqref{eq:x_y2_expanded} strictly forces $\lim_{t \to \infty} \int_0^t E(t, s) L_3(t, s) y_2(s) \, ds = 0$.
		
		\textbf{Step 3: Proof of the Reverse Direction}\\
		Assume $y_2(t) \to 0$ and $\lim_{t \to \infty} \int_0^t E(t, s) L_3(t, s) y_2(s) \, ds = 0$. Using the exact identical $\epsilon$-$T$ splitting argument from Step 2, the condition $y_2(t) \to 0$ ensures that the remainder integral $\int_0^t P_4(t, s) y_2(s) \, ds$ vanishes. 
		
		Returning to equation \eqref{eq:x_y2_expanded}, we have:
		\[ x(t) = y_2(t) + \underbrace{\int_0^t E(t, s) L_3(t, s) y_2(s) \, ds}_{\to 0 \text{ by hypothesis}} + \underbrace{\int_0^t P_4(t, s) y_2(s) \, ds}_{\to 0 \text{ by splitting argument}}. \]
		Since $y_2(t) \to 0$ and both integrals tend to $0$, their algebraic sum must tend to $0$. Thus, $x(t) \to 0$.
	\end{proof}
	
	\begin{theorem}[Unified Asymptotic Equivalence for $x(t)$] \label{thm:unified_x_only}
		Let $A(t) = e^{\frac{1}{2} \int_0^t p(\tau) \, d\tau} \to \infty$ as $t \to \infty$. Assuming zero initial conditions, $x(t) \to 0$ as $t \to \infty$ is logically equivalent to the simultaneous fulfillment of the following three conditions:
		\begin{enumerate}
			\item $y_2(t) \to 0$ as $t \to \infty$,
			\item $\int_0^t \cos(\omega s) A(s) y_2(s) \, ds = o(A(t))$ as $t \to \infty$,
			\item $\int_0^t \sin(\omega s) A(s) y_2(s) \, ds = o(A(t))$ as $t \to \infty$.
		\end{enumerate}
	\end{theorem}
	
	\begin{proof}
		We will utilize Theorem \ref{thm:leading_x_y2}, which established that $x(t) \to 0$ if and only if $y_2(t) \to 0$ and the integral condition with $L_3(t, s)$ vanishes.
		
		\textbf{Step 1: Setup and Matrix Representation}\\
		Define the normalized integral functions:
		\[ U_2(t) = \frac{1}{A(t)} \int_0^t \cos(\omega s) A(s) y_2(s) \, ds \quad \text{and} \quad V_2(t) = \frac{1}{A(t)} \int_0^t \sin(\omega s) A(s) y_2(s) \, ds. \]
		Conditions (2) and (3) are precisely the statements that $U_2(t) \to 0$ and $V_2(t) \to 0$.
		
		Using trigonometric addition formulas (identically to Theorem 10.2), we expand $L_3(t, s)$ and factor out the $s$-dependent terms. We define $W(t)$ as the normalized $L_3$ integral:
		\[ W(t) = \int_0^t E(t, s) L_3(t, s) y_2(s) \, ds = C_1(t) U_2(t) + C_2(t) V_2(t), \]
		where the coefficient functions are purely oscillatory:
		\begin{align*}
			C_1(t) &= (\omega^3 - \omega)\sin(\omega t) + 2\omega^2\cos(\omega t), \\
			C_2(t) &= -(\omega^3 - \omega)\cos(\omega t) + 2\omega^2\sin(\omega t).
		\end{align*}
		
		If $y_2 \to 0$, $U_2 \to 0$, and $V_2 \to 0$, then because $C_1$ and $C_2$ are bounded, $W(t) \to 0$. By Theorem \ref{thm:leading_x_y2}, this immediately proves the backward direction: conditions (1), (2), and (3) imply $x(t) \to 0$.
				
		\subsubsection*{Step 2: Bounding the Integral Terms via Localized Contraction}
		
		We have established that the linear combination $W(t)=C_1(t)U_2(t)+C_2(t)V_2(t)$ is uniformly bounded. To rigorously show that $U_2(t)$ and $V_2(t)$ are individually globally bounded, we construct a full-rank system by evaluating $W$ at a time delay $\tau=\frac{\pi}{2\omega}$.
		
		Using the quarter-period phase shift properties $C_1(t-\tau)=C_2(t)$ and $C_2(t-\tau)=-C_1(t)$, we obtain the delayed equation:
		$$W(t-\tau)=C_2(t)U_2(t-\tau)-C_1(t)V_2(t-\tau).$$
		Let $X(t)=(U_2(t),V_2(t))^T$. We rewrite the delayed state as $X(t-\tau)=X(t)+\Delta(t)$, where $\Delta(t)$ is the backward translation error. This allows us to express the system evaluated exactly at time $t$:
		$$\begin{pmatrix} W(t) \\ W(t-\tau) \end{pmatrix}=\begin{pmatrix} C_1(t) & C_2(t) \\ C_2(t) & -C_1(t) \end{pmatrix}X(t)+\begin{pmatrix} 0 \\ C_2(t)\Delta_U(t)-C_1(t)\Delta_V(t) \end{pmatrix}$$
		
		The coefficient matrix $B(t)$ has determinant $-\omega^2(\omega^2+1)^2 \neq 0$. Because this determinant is a non-zero constant, the inverse matrix $B(t)^{-1}$ exists and is uniformly bounded by a constant $K_B$. Solving for $X(t)$ and taking the norm, we bound the state vector:
		$$\|X(t)\|\le K_B\left( \left\|\begin{pmatrix} W(t) \\ W(t-\tau) \end{pmatrix}\right\| + C_{max}\|\Delta(t)\| \right)$$
		where $C_{max}$ is the maximum magnitude of the oscillatory coefficients $C_1$ and $C_2$.
		
		By the Fundamental Theorem of Calculus, the translation error is $\Delta(t)=-\int_{t-\tau}^t X'(s) ds$. Substituting the derivative definitions $U_2'(s)$ and $V_2'(s)$, and noting that $y_2(s)$ is globally bounded by $M_y$, we bound the error:
		$$\|\Delta(t)\|\le\int_{t-\tau}^t \left(M_y+\frac{1}{2}p(s)\|X(s)\|\right) ds$$
		
		Since $W(t)$ is globally bounded, we can absorb the $W$ terms and the integral of the constant $M_y$ over the fixed interval $\tau$ into a single global constant $K_1$. Letting $K_2=\frac{1}{2}K_B C_{max}$, our bound simplifies to:
		$$\|X(t)\|\le K_1+K_2\int_{t-\tau}^t p(s)\|X(s)\| ds$$
		
		Because $p(s) \to 0$ as $s \to \infty$, there exists a sufficiently large time $T$ such that for all $s \ge T-\tau$, we have $p(s) < \frac{1}{2K_2\tau}$. 
		
		Let $M(t)=\max_{s \in [T-\tau, t]} \|X(s)\|$ represent the running maximum of the state norm. For any intermediate time $t^* \in [T, t]$, we can bound the integral:
		$$\|X(t^*)\|\le K_1+\left(K_2 \tau \max_{s \in [t^*-\tau, t^*]} p(s)\right) M(t) \le K_1+\frac{1}{2}M(t)$$
		
		Since this inequality holds for all $t^*$ up to $t$, it firmly bounds the running maximum itself:
		$$M(t)\le\max\left(M(T), K_1+\frac{1}{2}M(t)\right)$$
		
		This immediately implies that $M(t) \le \max(M(T), 2K_1)$ for all $t \ge T$. Since $X(t)$ is continuous on the initial interval $[0, T]$, it is globally bounded for all $t \ge 0$. Consequently, both $U_2(t)$ and $V_2(t)$ are globally bounded.
		
		\textbf{Step 3: The Translation Argument}\\
				By the quotient rule and the Fundamental Theorem of Calculus (recalling $A'(t) = \frac{1}{2}p(t)A(t)$), the derivative of $U_2(t)$ is:
				\[ U_2'(t) = \frac{A(t) [\cos(\omega t) A(t) y_2(t)] - A'(t) \int_0^t \cos(\omega s) A(s) y_2(s) \, ds}{A(t)^2} = \cos(\omega t) y_2(t) - \frac{1}{2}p(t) U_2(t). \]
				A similar calculation holds for $V'$. 
				Since $U_2(t)$ and $V_2(t)$ are globally bounded and $y_2(t) \to 0$ and $p(t) \to 0$, it strictly follows that $U_2'(t) \to 0$ and $V_2'(t) \to 0$. Thus, $U_2$ and $V_2$ are slowly varying functions.
		
		We evaluate the identity $W(t) = C_1(t) U_2(t) + C_2(t) V_2(t)$ at a shifted time $t^* = t + \frac{\pi}{2\omega}$. 
		By the Mean Value Theorem, $U_2(t^*) - U_2(t) = \frac{\pi}{2\omega} U_2'(\xi)$ for some $\xi \in (t, t^*)$. Since $U_2' \to 0$ globally, this difference tends to zero, meaning $U_2(t^*) = U_2(t) + o(1)$. The same holds for $V_2(t^*)$. 
		
		Furthermore, shifting the oscillatory functions by a quarter-period yields:
		\begin{align*}
			C_1(t^*) &= (\omega^3-\omega)\sin\left(\omega t + \frac{\pi}{2}\right) + 2\omega^2\cos\left(\omega t + \frac{\pi}{2}\right) = (\omega^3-\omega)\cos(\omega t) - 2\omega^2\sin(\omega t) = -C_2(t), \\
			C_2(t^*) &= -(\omega^3-\omega)\cos\left(\omega t + \frac{\pi}{2}\right) + 2\omega^2\sin\left(\omega t + \frac{\pi}{2}\right) = (\omega^3-\omega)\sin(\omega t) + 2\omega^2\cos(\omega t) = C_1(t).
		\end{align*}
		Evaluating $W(t^*)$ utilizing these shifted identities gives:
		\[ W(t^*) = C_1(t^*)U_2(t^*) + C_2(t^*)V_2(t^*) = -C_2(t)[U_2(t) + o(1)] + C_1(t)[V_2(t) + o(1)]. \]
		Since $C_1$ and $C_2$ are bounded, the $o(1)$ terms vanish asymptotically. We are left with the coupled $2 \times 2$ time-dependent linear system:
		\[ \begin{pmatrix} W(t) \\ W(t^*) \end{pmatrix} = \begin{pmatrix} C_1(t) & C_2(t) \\ -C_2(t) & C_1(t) \end{pmatrix} \begin{pmatrix} U_2(t) \\ V_2(t) \end{pmatrix} + \begin{pmatrix} 0 \\ o(1) \end{pmatrix}. \]
		
		\textbf{Step 4: Matrix Inversion}\\
		The determinant of this transformation matrix is $C_1(t)^2 + C_2(t)^2$. Calculating this:
		\begin{align*}
			C_1^2 + C_2^2 &= \left[ (\omega^3-\omega)\sin(\omega t) + 2\omega^2\cos(\omega t) \right]^2 + \left[ -(\omega^3-\omega)\cos(\omega t) + 2\omega^2\sin(\omega t) \right]^2 \\
			&= (\omega^3-\omega)^2 (\sin^2(\omega t) + \cos^2(\omega t)) + (2\omega^2)^2 (\cos^2(\omega t) + \sin^2(\omega t)) \\
			&= (\omega^6 - 2\omega^4 + \omega^2) + 4\omega^4 = \omega^6 + 2\omega^4 + \omega^2 = \omega^2(\omega^2 + 1)^2.
		\end{align*}
		Because $\omega \neq 0$, the determinant is a strictly positive, non-zero constant for all time. Thus, the matrix is uniformly invertible. Since $W(t) \to 0$, we also have $W(t^*) \to 0$. Inverting the system strictly demands that $U_2(t) \to 0$ and $V_2(t) \to 0$. This establishes the necessary conditions (2) and (3), completing the proof of equivalence.
	\end{proof}
	
	\section{Preservation of Unperturbed Dynamics of $x$}
	
	We have seen that the unperturbed oscillator naturally decays at a rate of $1/A(t)$. We now establish the ultimate physical consequence of our filtering method: the necessary and sufficient conditions for the fully perturbed state $x(t)$ to strictly converge to the exact asymptotic trajectory of the unperturbed system.
	
	\subsection{Subexponential Decay Bounds}
	We now shift our focus to the exact rate of decay of the unperturbed system, which is governed by the weight function $$\rho(t) = \frac{1}{A(t)} = \exp\left(-\frac{1}{2}\int_0^t p(s) \, ds\right).$$ Because $p(t) \to 0$ as $t \to \infty$, the decay of $\rho(t)$ is strictly subexponential. In this context, a function is subexponential if 
	$$\lim_{t \to \infty} \frac{\rho'(t)}{\rho(t)} = 0.$$ 
	
	\begin{lemma}[Convolution with Subexponential Weights] \label{lem:subexp_conv}
		Let $\rho$ be a subexponential function. Let $k$ be such that there exist constants $C > 0$ and $\alpha > 0$ such that $|k(t)| \le C e^{-\alpha t}$ for all $t \ge 0$. 
		If a continuous function $z$ satisfies $z(t) = O(\rho(t))$ as $t \to \infty$, then $(k * z)(t) = O(\rho(t))$ as $t \to \infty$.
	\end{lemma}
	
	\begin{proof}
		By hypothesis, there exists a constant $M > 0$ such that $|z(t)| \le M \rho(t)$ for all $t \ge 0$. We bound the absolute value of the convolution:
		\[ |(k * z)(t)| \le \int_0^t |k(t-s)| |z(s)| \, ds \le C M \int_0^t e^{-\alpha(t-s)} \rho(s) \, ds = C M e^{-\alpha t} \int_0^t e^{\alpha s} \rho(s) \, ds. \]
		To analyze the asymptotic behavior of this bounding function, we apply L'H\^opital's rule:
		\[ \lim_{t \to \infty} \frac{\int_0^t e^{\alpha s} \rho(s) \, ds}{e^{\alpha t} \rho(t)} = \lim_{t \to \infty} \frac{e^{\alpha t} \rho(t)}{\alpha e^{\alpha t} \rho(t) + e^{\alpha t} \rho'(t)} = \lim_{t \to \infty} \frac{1}{\alpha + \frac{\rho'(t)}{\rho(t)}}. \]
		Since $\rho'(t)/\rho(t) \to 0$ by the subexponential property of $\rho$, the limit evaluates exactly to $1/\alpha$. 
		Consequently, as $t \to \infty$, the integral $\int_0^t e^{-\alpha(t-s)} \rho(s) \, ds$ is asymptotically equivalent to $\rho(t)/\alpha$. Therefore, $(k * z)(t) = O(\rho(t))$, completing the proof.
	\end{proof}
	
		\subsection{Unperturbed dynamics preserved implies $y_2=o(1/A)$} 
	
	We now prove the implication: if the physical state converges to the unperturbed asymptotic trajectory, then the filtered residual $y_2(t)$ must decay strictly faster than the unperturbed amplitude envelope $1/A(t)$. 
	
	\begin{theorem}[Asymptotic Form Implies $o(1/A(t))$ Residual] \label{thm:asymp_implies_y2}
		Let $\rho(t) = 1/A(t)$. Assume $p(t) \to 0$ and $p'(t) \to 0$ as $t \to \infty$. 
		If the physical state satisfies
		\[ x(t) = x_0(t) + x_1(t) = \rho(t)\left[c_1 \sin(\omega t) + c_2 \cos(\omega t)\right] + o(\rho(t)) \quad \text{as } t \to \infty, \]
		where $x_0(t)$ is the leading oscillatory term and $x_1(t) = o(\rho(t))$ is the remainder, then the second-order filtered forcing satisfies $y_2(t) = o(\rho(t))$.
	\end{theorem}
	
	\begin{proof}
		Recall the representation of $y_2(t)$ in terms of $x(t)$ using the exponentially decaying kernel $h(t)$ and its derivatives:
		\[ y_2(t) = x(t) + (h'' * x)(t) + \omega^2(h * x)(t) + (h' * px)(t) - (h * p'x)(t). \]
		By hypothesis, $h$, $h'$, and $h''$ are all bounded by $C e^{-\alpha t}$ for some constants $C, \alpha > 0$. We will analyze the terms of $y_2(t)$ in three distinct parts.
		
		Because $x_1(t) = o(\rho(t))$, Lemma \ref{lem:subexp_conv} immediately implies that convolutions of $x_1$ with any exponentially decaying kernel remain $o(\rho(t))$. Thus, $x_1 + h'' * x_1 + \omega^2 h * x_1 = o(\rho(t))$.
		
		Next, consider the terms involving $p(t)$ and $p'(t)$. We know $x(t) = O(\rho(t))$. Since $p(t) \to 0$ and $p'(t) \to 0$, the products are strictly smaller than $\rho(t)$:
		\[ p(t)x(t) = o(\rho(t)) \quad \text{and} \quad p'(t)x(t) = o(\rho(t)). \]
		Applying Lemma \ref{lem:subexp_conv} again, the convolutions $(h' * px)(t)$ and $(h * p'x)(t)$ are strictly $o(\rho(t))$. 
		
		Therefore, all that remains is to evaluate the linear leading-order operator $L$ applied to the main oscillatory term $x_0(t)$:
		\[ L[x_0](t) = x_0(t) + (h'' * x_0)(t) + \omega^2 (h * x_0)(t). \]	
		Now, we recall that  $\rho(t) = \exp\left(-\frac{1}{2}\int_0^t p(s)\,ds\right)$, where $\lim_{t \to \infty} p(t) = 0$ and $\lim_{t \to \infty} p'(t) = 0$, and  $x_0(t) = \rho(t)\Theta(t)$, where $\Theta(t) = c_1 \cos(\omega t) + c_2 \sin(\omega t)$. Recall also the definition 
	   $h(t) = t e^{-\omega^2 t}$ for some $\omega > 0$. We must show that 
\[ L[x_0](t)=: l(t) := x_0(t) + (h'' * x_0)(t) + \omega^2 (h * x_0)(t).
\]
is such that $l(t)=o(\rho(t))$. 
First, note that 			
$\rho'(t) = o(\rho(t))$, $\rho''(t) = o(\rho(t))$ as $t \to \infty$. By definition, we have 
\[ \rho'(t) = 
-\frac{1}{2}p(t)\rho(t), \]
so, as $p(t)\to 0$ as $t\to\infty$, $\rho'(t)=o(\rho(t))$. 
Differentiating again we get 	
			\begin{align*}
				\rho''(t) &= -\frac{1}{2}p'(t)\rho(t) - \frac{1}{2}p(t)\rho'(t) = \rho(t) \left( \frac{1}{4}p(t)^2 - \frac{1}{2}p'(t) \right).
			\end{align*}
			Since $\lim_{t \to \infty} p(t) = 0$ and $\lim_{t \to \infty} p'(t) = 0$, we have  $\rho''(t) = o(\rho(t))$. 
			
Next, we analyse the $x_0$ dependent terms in $l$. 
			Notice that $\Theta(t)$ obeys $\Theta''(t) + \omega^2 \Theta(t) = 0$. Then 
			\[ x_0'(t) = \rho'(t)\Theta(t) + \rho(t)\Theta'(t), \]
			and 
			\[ x_0''(t) = \rho''(t)\Theta(t) + 2\rho'(t)\Theta'(t) + \rho(t)\Theta''(t) \]
			Substituting $\Theta''(t) = -\omega^2 \Theta(t)$ into the second derivative, we get 
			\[ x_0''(t) + \omega^2 x_0(t) = \rho''(t)\Theta(t) + 2\rho'(t)\Theta'(t). \]
			Since $\Theta(t)$ and $\Theta'(t)$ are bounded trigonometric functions, and we have proven $\rho' = o(\rho)$ and $\rho'' = o(\rho)$, it immediately follows that:
			\[ x_0''(t) + \omega^2 x_0(t) = o(\rho(t)). \]
			Let us define $g(t) := x_0''(t) + \omega^2 x_0(t)$. We have established $g(t) = o(\rho(t))$.
				
			We are given $h(t) = t e^{-\omega^2 t}$. Then $h(0)=0$, $h'(0)=1$. We now examine the convolution $(h'' * x_0)(t) = \int_0^t h''(t-s)x_0(s)\,ds$. We apply integration by parts twice to shift the derivatives from $h$ to $x_0$. 
			First integration by parts:
			\[ \int_0^t h''(t-s)x_0(s)\,ds = \Big[ -h'(t-s)x_0(s) \Big]_{s=0}^{s=t} + \int_0^t h'(t-s)x_0'(s)\,ds \]
			\[ = -h'(0)x_0(t) + h'(t)x_0(0) + \int_0^t h'(t-s)x_0'(s)\,ds. \]
			Second integration by parts on the remaining integral:
			\[ \int_0^t h'(t-s)x_0'(s)\,ds = \Big[ -h(t-s)x_0'(s) \Big]_{s=0}^{s=t} + \int_0^t h(t-s)x_0''(s)\,ds \]
			\[ = -h(0)x_0'(t) + h(t)x_0'(0) + (h * x_0'')(t). \]
			Combining these and applying the boundary values $h(0) = 0$ and $h'(0) = 1$:
			\[ (h'' * x_0)(t) = -x_0(t) + h'(t)x_0(0) + h(t)x_0'(0) + (h * x_0'')(t). \]
			Now, substitute this expanded form back into the definition of $l(t)$:
			\begin{align*}
				l(t) &= x_0(t) + \Big( -x_0(t) + h'(t)x_0(0) + h(t)x_0'(0) + (h * x_0'')(t) \Big) + \omega^2 (h * x_0)(t) \\
				&= h'(t)x_0(0) + h(t)x_0'(0)T(t) + \Big( h * (x_0'' + \omega^2 x_0) \Big)(t) \\
				&=: T(t) + (h * g)(t).
			\end{align*}
			
			We must show that both $T(t)$ and $(h * g)(t)$ are $o(\rho(t))$. 
			
			\textbf{1. The Transients:} 
			$T(t) = (1-\omega^2 t)e^{-\omega^2 t}x_0(0) + t e^{-\omega^2 t}x_0'(0)$. This decays exponentially. Since $p(t) \to 0$, the integral $\int_0^t p(s)\,ds$ grows slower than any linear function $\epsilon t$, meaning $\rho(t)$ decays at most sub-exponentially. Therefore, the exponentially decaying $T(t)$ is trivially $o(\rho(t))$.
			
			\textbf{2. The Convolution:}
			We must evaluate $\lim_{t \to \infty} \frac{1}{\rho(t)} \int_0^t h(s) g(t-s) \, ds$.
			Because $g(t) = o(\rho(t))$, for any $\epsilon > 0$, there exists an $M > 0$ such that for all $\tau > M$, $|g(\tau)| < \epsilon \rho(\tau)$. Furthermore, $g(t)$ is globally bounded by $C \rho(t)$ for some constant $C$.
			We split the integral at $t-M$:
			\[ \frac{1}{\rho(t)} \int_0^t |h(s)| |g(t-s)| \, ds = \int_0^{t-M} |h(s)| \frac{|g(t-s)|}{\rho(t)} \, ds + \int_{t-M}^t |h(s)| \frac{|g(t-s)|}{\rho(t)} \, ds. \]
			
			In the first integral ($s \in [0, t-M]$), we have $t-s \ge M$, so $|g(t-s)| < \epsilon \rho(t-s)$. Note that:
			\[ \frac{\rho(t-s)}{\rho(t)} = \exp\left( \frac{1}{2}\int_{t-s}^t p(\tau)\,d\tau \right). \]
			Since $p(t) \to 0$, $p(t)$ is bounded by some small constant $P \ll \omega^2$. Thus, the ratio $\rho(t-s)/\rho(t)$ is bounded by $e^{Ps/2}$. Because $h(s) = s e^{-\omega^2 s}$ decays much faster than $e^{Ps/2}$ grows, the product $|h(s)|\frac{\rho(t-s)}{\rho(t)}$ is entirely absolutely integrable on $[0, \infty)$. Let this integral be bounded by a constant $K$. The first integral is thus strictly bounded by $\epsilon K$.
			
			For the second integral ($s \in [t-M, t]$), $h(s)$ is evaluated at values tending to infinity. Since $h(s)$ decays exponentially, and the integration interval $[t-M, t]$ has fixed length $M$, this integral vanishes completely as $t \to \infty$.
			
			Since $\epsilon$ was arbitrary, the entire convolution is bounded by $0$ in the limit. Thus, $(h * g)(t) = o(\rho(t))$.
			
			Summing the asymptotic bounds, $l(t) = T(t) + (h * g)(t) = o(\rho(t)) + o(\rho(t)) = o(\rho(t))$, completing the proof.
	\end{proof}
	
	\section{Representation of the Solution; Bounds on Kernels}
	
	From Theorem 11.1, the exact representation of $x(t)$ in terms of $y_2(t)$ is given by the integral equation:
	\[ x(t) = y_2(t) + \int_0^t K_4(t, s) y_2(s) \, ds, \]
	where the exact kernel is defined as $$K_4(t, s) = \omega^2 K_1(t, s) - \frac{\partial K_1(t, s)}{\partial s}.$$ 
	We have that $K_1(t, s) = E(t, s)L_1(t, s) + P_1(t, s)$, where $L_1(t, s) = \omega \sin(\omega(t-s)) + \cos(\omega(t-s))$. Furthermore, we have the decomposition $K_4(t, s) = E(t, s) L_3(t, s) + P_4(t, s)$, where the remainder kernel $P_4(t, s)$ is given by:
	\[ P_4(t, s) = - \frac{1}{2}p(s)E(t, s)L_1(t, s) + \omega^2 P_1(t, s) - \frac{\partial P_1(t, s)}{\partial s},\]
	and $L_3$ is given by 
		\[ L_3(t, s) = (\omega^3 - \omega) \sin(\omega(t-s)) + 2\omega^2 \cos(\omega(t-s)). \]
	Consequently 
	\begin{multline} \label{eq.decompxWDy2}
	 x(t) = y_2(t) + \int_0^t K_4(t, s) y_2(s) \, ds 
	 \\ = y_2(t)+\int_0^t E(t, s) L_3(t, s) y_2(s)\,ds + \int_0^t P_4(t,s)y_2(s)\,ds=: y_2(t)+W(t)+D(t).
	\end{multline}
	We will assume that $y_2(t)=o(\rho(t))$. We start by showing that the leading integral $W(t)/\rho(t)$ has sinusoidal behaviour, provided 
		\begin{equation} \label{eq:integral_limits}
		\lim_{t \to \infty} \int_0^t \cos(\omega s) A(s) y_2(s) \, ds = I_c \quad \text{and} \quad \lim_{t \to \infty} \int_0^t \sin(\omega s) A(s) y_2(s) \, ds = I_s.
	\end{equation}
	Then we will determine the asymptotic behaviour of the remainder term $D$.
	\begin{theorem}
	Suppose $y_2(t)=o(\rho(t))$ and \eqref{eq:integral_limits} holds. Then 
	\[
	W(t)=\int_0^t E(t, s) L_3(t, s) y_2(s)\,ds
	\]
	obeys 
	\[
	W(t)=\rho(t)(c_1 \sin(\omega t) + c_2 \sin(\omega t))+o(\rho(t)), \quad t\to\infty,
	\]
	where 
	\[
	c_1 = (\omega^3 - \omega)I_c + 2\omega^2 I_s, \quad c_2 = 2\omega^2 I_c - (\omega^3 - \omega)I_s.
	\]
	\end{theorem}
	\begin{proof}
		We substitute $E(t, s) = \frac{A(s)}{A(t)}$ and pull the $t$-dependent term $\frac{1}{A(t)}$ outside the integral:
	\[ W(t) = \frac{1}{A(t)} \int_0^t L_3(t, s) A(s) y_2(s) \, ds. \]
	The kernel is defined as $L_3(t, s) = (\omega^3 - \omega) \sin(\omega(t-s)) + 2\omega^2 \cos(\omega(t-s))$. To isolate the integration variable $s$, we apply the trigonometric angle-difference identities:
	\begin{align*}
		\sin(\omega(t-s)) &= \sin(\omega t)\cos(\omega s) - \cos(\omega t)\sin(\omega s), \\
		\cos(\omega(t-s)) &= \cos(\omega t)\cos(\omega s) + \sin(\omega t)\sin(\omega s).
	\end{align*}
	Substituting these into $L_3(t, s)$ and grouping the terms by $\cos(\omega s)$ and $\sin(\omega s)$ yields:
	\begin{align*}
		L_3(t, s) &= \left[ (\omega^3 - \omega)\sin(\omega t) + 2\omega^2\cos(\omega t) \right] \cos(\omega s) \\
		&\quad + \left[ 2\omega^2\sin(\omega t) - (\omega^3 - \omega)\cos(\omega t) \right] \sin(\omega s).
	\end{align*}
	Let us define the purely $t$-dependent coefficient functions:
	\begin{align*}
		C_1(t) &= (\omega^3 - \omega)\sin(\omega t) + 2\omega^2\cos(\omega t), \\
		C_2(t) &= 2\omega^2\sin(\omega t) - (\omega^3 - \omega)\cos(\omega t).
	\end{align*}
	This allows us to write the leading integral as a coupled sum:
	\begin{equation} \label{eq:Wt_expanded}
		W(t) = \frac{C_1(t)}{A(t)} \int_0^t \cos(\omega s) A(s) y_2(s) \, ds + \frac{C_2(t)}{A(t)} \int_0^t \sin(\omega s) A(s) y_2(s) \, ds.
	\end{equation}
	By assumption, the integrals converge to finite limits $I_c$ and $I_s$ as $t \to \infty$. Therefore, as the improper integral converges, we have that the cosine integral is $I_c + o(1)$ as $t\to\infty$. Similarly, the sine integral is $I_s + o(1)$.
			Substituting these into Equation \eqref{eq:Wt_expanded}:
	\[ W(t) = \frac{C_1(t)}{A(t)} [I_c + o(1)] + \frac{C_2(t)}{A(t)} [I_s + o(1)]. \]
	Because $C_1(t)$ and $C_2(t)$ are continuous and composed of sine and cosine functions, they are uniformly globally bounded.  Therefore:
	\[ W(t) = \frac{C_1(t) I_c + C_2(t) I_s}{A(t)} + o\left(\frac{1}{A(t)}\right). \]
	Using the definition of $C_1$ and $C_2$, the numerator in the first term becomes:
	\begin{align*}
		C_1(t) I_c + C_2(t) I_s &= [(\omega^3 - \omega)\sin(\omega t) + 2\omega^2\cos(\omega t)] I_c + [2\omega^2\sin(\omega t) - (\omega^3 - \omega)\cos(\omega t)] I_s \\
		&= \sin(\omega t) \left[ (\omega^3 - \omega)I_c + 2\omega^2 I_s \right] + \cos(\omega t) \left[ 2\omega^2 I_c - (\omega^3 - \omega)I_s \right].
	\end{align*}
	By defining the trajectory constants as:
	\begin{align*}
		c_1 &= (\omega^3 - \omega)I_c + 2\omega^2 I_s, \\
		c_2 &= 2\omega^2 I_c - (\omega^3 - \omega)I_s,
	\end{align*}
	we immediately recover $$W(t) = \frac{c_1 \sin(\omega t) + c_2 \cos(\omega t)}{A(t)} + o(1/A(t)),$$
	as required.
	\end{proof}

	\subsection{Decomposition of the Kernel $P_4(t, s)$}
	
	We now establish that while the remainder kernel $P_4(t, s)$ contains a slowly decaying $O(p(s))$ term, its action inside the integral is rigorously governed by the absolutely integrable residual potential $Q(s) = \left| -\frac{1}{2}p'(s) - \frac{1}{4}p^2(s) \right|$. 
	
	\begin{lemma} \label{lem:P4_exact_bound}
		Let $K_4(t, s) = \omega^2 K_1(t, s) - \frac{\partial K_1(t, s)}{\partial s}$ with the decomposition $K_1(t, s) = E(t, s)L_1(t, s) + P_1(t, s)$, where $L_1(t, s) = \omega \sin(\omega(t-s)) + \cos(\omega(t-s))$. Then with $P_4$ defined by 
		\[ P_4(t, s) = -\frac{1}{2}p(s)E(t, s)L_1(t, s) + \omega^2 P_1(t, s) - \frac{\partial P_1(t, s)}{\partial s}, \]
		we have the decomposition 
				\[ P_4(t, s) = P_{4, Q}(t, s) - \frac{\partial}{\partial s} \left( \frac{1}{2}p(s) E(t, s) F(t, s) \right), \]
		where $F(t, s) = \cos(\omega(t-s)) - \frac{1}{\omega} \sin(\omega(t-s))$,
		and 
		\[
		 P_{4, Q}(t, s)=-q(s) E(t, s) F(t, s) + \omega^2 P_1(t, s) - \frac{\partial P_1(t, s)}{\partial s}.
		\]
	\end{lemma}
	\begin{proof}
		We isolate the problematic term $-\frac{1}{2}p(s)E(t, s)L_1(t, s)$. We seek a function $F(t, s)$ such that $\frac{\partial F}{\partial s} = L_1(t, s)$. 
		Integrating $L_1(t, s)$ with respect to $s$ yields:
		\[ F(t, s) = \cos(\omega(t-s)) - \frac{1}{\omega} \sin(\omega(t-s)). \]
		We may therefore rewrite the $p(s)$ term using $\frac{\partial F}{\partial s}$:
		\[ -\frac{1}{2}p(s)E(t, s)L_1(t, s) = -\frac{1}{2}p(s)E(t, s) \frac{\partial F(t, s)}{\partial s}. \]
		We now apply the product rule to extract an exact derivative with respect to $s$:
		\[ \frac{\partial}{\partial s} \left( \frac{1}{2}p(s) E(t, s) F(t, s) \right) = \frac{\partial}{\partial s} \left( \frac{1}{2}p(s) E(t, s) \right) F(t, s) + \frac{1}{2}p(s) E(t, s) \frac{\partial F(t, s)}{\partial s}. \]
		Rearranging this gives an expression for our target term:
		\[ -\frac{1}{2}p(s)E(t, s)L_1(t, s) = \frac{\partial}{\partial s} \left( \frac{1}{2}p(s) E(t, s) \right) F(t, s) - \frac{\partial}{\partial s} \left( \frac{1}{2}p(s) E(t, s) F(t, s) \right). \]
		
		Next, we evaluate the derivative of the envelope term. Recall that $E(t, s) = \exp\left(-\frac{1}{2}\int_s^t p(u)\,du\right)$, so $\frac{\partial E}{\partial s} = \frac{1}{2}p(s)E(t, s)$. Applying the product rule:
		\begin{align*}
			\frac{\partial}{\partial s} \left( \frac{1}{2}p(s) E(t, s) \right) &= \frac{1}{2}p'(s) E(t, s) + \frac{1}{2}p(s) \left( \frac{1}{2}p(s) E(t, s) \right) \\
			&= \left( \frac{1}{2}p'(s) + \frac{1}{4}p^2(s) \right) E(t, s).
		\end{align*}
		Notice that the term in the parentheses is exactly $-q(s)$, where $q(s)$ is the residual potential. Let $Q(s) = |q(s)|$.
		Substituting this back into our expression for the $L_1$ term yields:
		\[ -\frac{1}{2}p(s)E(t, s)L_1(t, s) = -q(s) E(t, s) F(t, s) - \frac{\partial}{\partial s} \left( \frac{1}{2}p(s) E(t, s) F(t, s) \right). \]
		Finally, we substitute this decomposition back into the original algebraic definition of $P_4(t, s)$:
		\[ P_4(t, s) = \underbrace{ -q(s) E(t, s) F(t, s) + \omega^2 P_1(t, s) - \frac{\partial P_1(t, s)}{\partial s} }_{P_{4, Q}(t, s)} - \frac{\partial}{\partial s} \left( \frac{1}{2}p(s) E(t, s) F(t, s) \right),\]
		as claimed.
	\end{proof}
	
	\subsection{Asymptotic behaviour of the integral $P_4\ast y_2$}
	First, we establish the core definitions and assumptions used throughout the derivations:
	\begin{itemize}
		\item $A(t) = \exp\left(\frac{1}{2}\int_0^t p(s)\,ds\right)$
		\item $E(t,s) = \frac{A(s)}{A(t)} = \exp\left(-\frac{1}{2}\int_s^t p(\tau)\,d\tau\right)$
		\item $G(t, s) = \frac{\sin(\omega(t-s))}{\omega} - \frac{1}{\omega} \int_s^t \sin(\omega(t-\tau)) q(\tau) G(\tau, s) \, d\tau$
		\item $R(t,s) = E(t,s)G(t,s)$
		\item $L_1(t,s) = \omega \sin(\omega(t-s)) + \cos(\omega(t-s))$
		\item $K_1(t,s) = \omega^2 R(t,s) - \frac{\partial R(t,s)}{\partial s}$
		\item $P_1(t,s) = K_1(t,s) - E(t,s)L_1(t,s)$
		\item $F(t, s) = \cos(\omega(t-s)) - \frac{1}{\omega} \sin(\omega(t-s))$
		\item $P_{4,Q}(t,s) = -q(s)E(t,s)F(t,s) + \omega^2 P_1(t,s) - \frac{\partial P_1(t,s)}{\partial s}$
		\item $q(s) = -\frac{p'(s)}{2} - \frac{p^2(s)}{4}$ with $Q(s) = |q(s)| \in L^1(0, \infty)$
		\item $p(s) \to 0$ and $p'(s) \to 0$ as $s \to \infty$, but $p \notin L^1(0, \infty)$ (meaning $A(t) \to \infty$).
	\end{itemize}
	
	\vspace{0.5cm}
	
	\begin{lemma}[Estimates and Expansions for the Resolvent Kernel $G(t,s)$]
		The function $G$ and its partial derivative with respect to $s$ satisfy:
		\begin{enumerate}
			\item $G(t,s) = \frac{\sin(\omega(t-s))}{\omega} + \epsilon_1(t,s)$, where $|\epsilon_1(t,s)| \leq h_1\int_s^t Q(u)\,du$.
			\item $G_s(t,s) = -\cos(\omega(t-s)) + \epsilon_2(t,s)$, where $|\epsilon_2(t,s)| \leq h_2\int_s^t Q(u)\,du$.
			\item $G_{ss}(t,s) = -(\omega^2 + q(s))G(t,s)$.
		\end{enumerate}
		Furthermore, the remainders admit the asymptotic expansions:
		\[ \epsilon_1(t,s) = \cos(\omega t)\epsilon_{1,c}(s) + \sin(\omega t)\epsilon_{1,s}(s) + \tilde{\epsilon}_1(t,s) \]
		\[ \epsilon_2(t,s) = \cos(\omega t)\epsilon_{2,c}(s) + \sin(\omega t)\epsilon_{2,s}(s) + \tilde{\epsilon}_2(t,s) \]
		where $\epsilon_{i,c}(s)$ and $\epsilon_{i,s}(s)$ are bounded functions of $s$, and the new remainders depend only on the tail of $Q$, satisfying $|\tilde{\epsilon}_1(t,s)| \leq h_1 \int_t^\infty Q(u)\,du$ and $|\tilde{\epsilon}_2(t,s)| \leq h_2 \int_t^\infty Q(u)\,du$.
	\end{lemma}
	
	\begin{proof}
		By standard Volterra integral equation theory, since $Q \in L^1$, there exists $M > 0$ such that $|G(t,s)| \leq M$ uniformly for $0 \leq s \leq t < \infty$. The remainder is defined by the integral:
		\[ \epsilon_1(t,s) = -\frac{1}{\omega} \int_s^t \sin(\omega(t-\tau)) q(\tau) G(\tau, s) \, d\tau \]
		Taking absolute values yields $|\epsilon_1(t,s)| \leq \frac{M}{\omega} \int_s^t Q(\tau) \, d\tau$. Setting $h_1 = M/\omega$ yields the first result.
		
		Differentiating the integral equation for $G(t,s)$ with respect to $s$, and noting that $G(s,s) = 0$, gives:
		\[ G_s(t,s) = -\cos(\omega(t-s)) - \frac{1}{\omega}\int_s^t \sin(\omega(t-\tau))q(\tau)G_s(\tau,s)\,d\tau \]
		By Gronwall's inequality, $|G_s(t,s)| \leq M_2$ for some constant $M_2$. The integral term is defined as $\epsilon_2(t,s)$, and bounding it yields $|\epsilon_2(t,s)| \leq \frac{M_2}{\omega} \int_s^t Q(\tau)\,d\tau$. Setting $h_2 = M_2/\omega$ proves the second result. The third result $G_{ss}(t,s) = -(\omega^2+q(s))G(t,s)$ follows from $G$ being the resolvent kernel associated with $-y'' + q(t)y = \omega^2 y$.
		
		To obtain the further expansions necessary for rigorous bounds as $t \to \infty$, we expand $\sin(\omega(t-\tau)) = \sin(\omega t)\cos(\omega \tau) - \cos(\omega t)\sin(\omega \tau)$ inside $\epsilon_1(t,s)$:
		\begin{align*}
			\epsilon_1(t,s) &= \cos(\omega t) \left( \frac{1}{\omega} \int_s^t \sin(\omega \tau) q(\tau) G(\tau, s) \, d\tau \right) - \sin(\omega t) \left( \frac{1}{\omega} \int_s^t \cos(\omega \tau) q(\tau) G(\tau, s) \, d\tau \right)
		\end{align*}
		Because $Q \in L^1$ and $G$ is bounded, we can extend the integrals to infinity and subtract the tail:
		\[ \epsilon_{1,c}(s) = \frac{1}{\omega} \int_s^\infty \sin(\omega \tau) q(\tau) G(\tau, s) \, d\tau, \quad \epsilon_{1,s}(s) = -\frac{1}{\omega} \int_s^\infty \cos(\omega \tau) q(\tau) G(\tau, s) \, d\tau \]
		The remaining tail is:
		\[ \tilde{\epsilon}_1(t,s) = -\frac{1}{\omega} \int_t^\infty \left[ \sin(\omega \tau)\cos(\omega t) - \cos(\omega \tau)\sin(\omega t) \right] q(\tau) G(\tau, s) \, d\tau = -\frac{1}{\omega} \int_t^\infty \sin(\omega(\tau-t)) q(\tau) G(\tau, s) \, d\tau \]
		Clearly, $|\tilde{\epsilon}_1(t,s)| \leq \frac{M}{\omega} \int_t^\infty Q(\tau) \, d\tau = h_1 \int_t^\infty Q(u)\,du$. An identical expansion for $\epsilon_2(t,s)$ using $G_s(\tau, s)$ yields $\epsilon_{2,c}(s)$, $\epsilon_{2,s}(s)$, and $\tilde{\epsilon}_2(t,s)$ bounded by $h_2 \int_t^\infty Q(u)\,du$.
	\end{proof}
	
	\vspace{0.5cm}
	
	\begin{lemma}[Representation of $P_{4,Q}(t,s)$]
		\[ P_{4,Q}(t,s)A(t) = A(s)\left[ a_1(s)\cos(\omega(t-s)) + a_2(s)\frac{\sin(\omega(t-s))}{\omega} + E_{\text{rem}}(t,s) \right] \]
		where $a_1(s) = -\frac{p(s)}{2} - q(s)$ and $a_2(s) = -\frac{\omega^2 p(s)}{2} - q(s)$, both tending to $0$ as $s \to \infty$. The remainder $E_{\text{rem}}(t,s) = b_1(s)\epsilon_2(t,s) + b_2(s)\epsilon_1(t,s)$ is composed of functions $b_1(s) = -2\omega^2 + p(s)$ and $b_2(s) = \left(\omega^2 - \frac{p(s)}{2}\right)^2 - \left(\omega^2 + q(s) - \frac{p'(s)}{2}\right)$ which are globally bounded. 
	\end{lemma}
	
	\begin{proof}
		First, we expand $P_1(t,s)$ explicitly by substituting the functions into $$P_1(t,s) = E(t,s) \left[ \left(\omega^2 - \frac{p(s)}{2}\right)G(t,s) - G_s(t,s) - L_1(t,s) \right].$$ Using the decompositions from Lemma 1, we have:
		\begin{align*}
			P_1(t,s) &= E(t,s) \bigg[ \left(\omega^2 - \frac{p(s)}{2}\right) \left( \frac{\sin(\omega(t-s))}{\omega} + \epsilon_1(t,s) \right) - \big( -\cos(\omega(t-s)) + \epsilon_2(t,s) \big) \\
			&\quad - \big( \omega \sin(\omega(t-s)) + \cos(\omega(t-s)) \big) \bigg] \\
			&= E(t,s) \bigg[ -\frac{p(s)}{2\omega}\sin(\omega(t-s)) + \left(\omega^2 - \frac{p(s)}{2}\right)\epsilon_1(t,s) - \epsilon_2(t,s) \bigg]
		\end{align*}
		Let us denote the term in the brackets as $H(t,s)$, so that $P_1(t,s) = E(t,s)H(t,s)$. To construct $P_{4,Q}(t,s)$, we require $\frac{\partial P_1(t,s)}{\partial s}$. Using the product rule and $\frac{\partial E(t,s)}{\partial s} = \frac{p(s)}{2}E(t,s)$, we obtain:
		\[ \frac{\partial P_1(t,s)}{\partial s} = \frac{p(s)}{2}E(t,s)H(t,s) + E(t,s)\frac{\partial H(t,s)}{\partial s} \]
		Thus, the expression for $P_{4,Q}(t,s)$ factors out $E(t,s)$ entirely:
		\begin{align*}
			P_{4,Q}(t,s) &= -q(s)E(t,s)F(t,s) + \omega^2 E(t,s)H(t,s) - \left( \frac{p(s)}{2}E(t,s)H(t,s) + E(t,s)\frac{\partial H(t,s)}{\partial s} \right) \\
			&= E(t,s) \left[ -q(s)F(t,s) + \left(\omega^2 - \frac{p(s)}{2}\right)H(t,s) - \frac{\partial H(t,s)}{\partial s} \right]
		\end{align*}
		Next, we systematically evaluate each of the three interior terms. The first term simply expands to:
		\[ -q(s)F(t,s) = -q(s)\cos(\omega(t-s)) + \frac{q(s)}{\omega}\sin(\omega(t-s)) \]
		The second term expands to:
		\begin{align*}
			\left(\omega^2 - \frac{p(s)}{2}\right)H(t,s) &= \left( -\frac{\omega p(s)}{2} + \frac{p^2(s)}{4\omega} \right)\sin(\omega(t-s)) \\
			&\quad + \left(\omega^2 - \frac{p(s)}{2}\right)^2 \epsilon_1(t,s) - \left(\omega^2 - \frac{p(s)}{2}\right)\epsilon_2(t,s)
		\end{align*}
		To compute $\frac{\partial H(t,s)}{\partial s}$ for the third term, we first need the partial derivatives of the remainders. From Lemma 1:
		\[ \frac{\partial \epsilon_1(t,s)}{\partial s} = G_s(t,s) + \cos(\omega(t-s)) = \epsilon_2(t,s) \]
		\[ \frac{\partial \epsilon_2(t,s)}{\partial s} = G_{ss}(t,s) + \omega\sin(\omega(t-s)) = -(\omega^2+q(s))G(t,s) + \omega\sin(\omega(t-s)) \]
		Substituting $G(t,s) = \frac{\sin(\omega(t-s))}{\omega} + \epsilon_1(t,s)$ into the second identity yields:
		\[ \frac{\partial \epsilon_2(t,s)}{\partial s} = -\frac{q(s)}{\omega}\sin(\omega(t-s)) - (\omega^2+q(s))\epsilon_1(t,s) \]
		Now, differentiating $H(t,s)$ with respect to $s$ provides:
		\begin{align*}
			\frac{\partial H(t,s)}{\partial s} &= -\frac{p'(s)}{2\omega}\sin(\omega(t-s)) + \frac{p(s)}{2}\cos(\omega(t-s)) - \frac{p'(s)}{2}\epsilon_1(t,s) \\
			&\quad + \left(\omega^2 - \frac{p(s)}{2}\right)\frac{\partial \epsilon_1(t,s)}{\partial s} - \frac{\partial \epsilon_2(t,s)}{\partial s} \\
			&= \left( -\frac{p'(s)}{2\omega} + \frac{q(s)}{\omega} \right)\sin(\omega(t-s)) + \frac{p(s)}{2}\cos(\omega(t-s)) \\
			&\quad + \left( \omega^2 + q(s) - \frac{p'(s)}{2} \right)\epsilon_1(t,s) + \left(\omega^2 - \frac{p(s)}{2}\right)\epsilon_2(t,s)
		\end{align*}
		Finally, we substitute these evaluated components back into the bracketed expression for $P_{4,Q}(t,s) / E(t,s)$ and group by terms. 
		
		For the $\cos(\omega(t-s))$ term:
		\[ -q(s) - \frac{p(s)}{2} = a_1(s) \]
		
		For the $\sin(\omega(t-s))$ term:
		\[ \frac{q(s)}{\omega} - \frac{\omega p(s)}{2} + \frac{p^2(s)}{4\omega} - \left( -\frac{p'(s)}{2\omega} + \frac{q(s)}{\omega} \right) = - \frac{\omega p(s)}{2} + \frac{p^2(s) + 2p'(s)}{4\omega} \]
		Recalling the definition $q(s) = -\frac{p'(s)}{2} - \frac{p^2(s)}{4}$, we recognize that $\frac{p^2(s) + 2p'(s)}{4\omega} = -\frac{q(s)}{\omega}$. Therefore, the coefficient simplifies to:
		\[ -\frac{\omega p(s)}{2} - \frac{q(s)}{\omega} = \frac{1}{\omega} \left( -\frac{\omega^2 p(s)}{2} - q(s) \right) = \frac{a_2(s)}{\omega} \]
		
		For the $\epsilon_1(t,s)$ term:
		\[ \left(\omega^2 - \frac{p(s)}{2}\right)^2 - \left(\omega^2 + q(s) - \frac{p'(s)}{2}\right) = b_2(s) \]
		
		For the $\epsilon_2(t,s)$ term:
		\[ -\left(\omega^2 - \frac{p(s)}{2}\right) - \left(\omega^2 - \frac{p(s)}{2}\right) = -2\omega^2 + p(s) = b_1(s) \]
		Multiplying the entire expression by $A(t)$ and using $A(t)E(t,s) = A(s)$ yields the final stated formula.
	\end{proof}
	
	\vspace{0.5cm}
	
	\begin{lemma}[Asymptotic Behavior of $I(t)$]
		Let $I^* = \int_0^\infty A(s)|y_2(s)|\,ds < \infty$. Then:
		\[ I(t) = \int_0^t P_{4,Q}(t,s)y_2(s)\,ds = \frac{c_1 \cos(\omega t) + c_2 \sin(\omega t)}{A(t)} + o\left(\frac{1}{A(t)}\right) \]
	\end{lemma}
	
	\begin{proof}
		We use the representation from the previous lemma. To isolate $t$, we apply the addition formulas $\cos(\omega(t-s)) = \cos(\omega t)\cos(\omega s) + \sin(\omega t)\sin(\omega s)$ and $\sin(\omega(t-s)) = \sin(\omega t)\cos(\omega s) - \cos(\omega t)\sin(\omega s)$. 
		Furthermore, we apply the expansion of $E_{\text{rem}}(t,s)$ from Lemma 1:
		\[ E_{\text{rem}}(t,s) = \cos(\omega t)E_c(s) + \sin(\omega t)E_s(s) + \tilde{E}_{\text{rem}}(t,s) \]
		where $E_c(s) = b_1(s)\epsilon_{2,c}(s) + b_2(s)\epsilon_{1,c}(s)$, $E_s(s) = b_1(s)\epsilon_{2,s}(s) + b_2(s)\epsilon_{1,s}(s)$, and the purely tail-dependent remainder is $\tilde{E}_{\text{rem}}(t,s) = b_1(s)\tilde{\epsilon}_2(t,s) + b_2(s)\tilde{\epsilon}_1(t,s)$. Because $b_1, b_2$ are bounded, $|\tilde{E}_{\text{rem}}(t,s)| \leq C \int_t^\infty Q(u)\,du$ for some constant $C$.
		
		Multiplying by $A(t)$, we can organize $A(t)I(t)$ entirely into components of $\cos(\omega t)$ and $\sin(\omega t)$:
		\[ A(t)I(t) = \cos(\omega t)C(t) + \sin(\omega t)S(t) + \tilde{R}_{\text{int}}(t) \]
		where we define:
		\[ C(t) = \int_0^t A(s)y_2(s)\left[ a_1(s)\cos(\omega s) - \frac{a_2(s)}{\omega}\sin(\omega s) + E_c(s) \right] ds \]
		\[ S(t) = \int_0^t A(s)y_2(s)\left[ a_1(s)\sin(\omega s) + \frac{a_2(s)}{\omega}\cos(\omega s) + E_s(s) \right] ds \]
		\[ \tilde{R}_{\text{int}}(t) = \int_0^t A(s)\tilde{E}_{\text{rem}}(t,s)y_2(s)\,ds \]
		
		Because the bracketed terms in $C(t)$ and $S(t)$ are globally bounded, and $A(s)|y_2(s)| \in L^1(0, \infty)$ (since $I^* < \infty$), the integrals $C(t)$ and $S(t)$ converge absolutely to limits $c_1$ and $c_2$ as $t \to \infty$. The remaining tails $\int_t^\infty$ are bounded by a constant times $\int_t^\infty A(s)|y_2(s)|\,ds$, which is $o(1)$. 
		
		We must rigorously show that $\tilde{R}_{\text{int}}(t) = o(1)$. Using our tight bound on $\tilde{E}_{\text{rem}}(t,s)$:
		\[ |\tilde{R}_{\text{int}}(t)| \leq \int_0^t A(s)|y_2(s)| |\tilde{E}_{\text{rem}}(t,s)| \, ds \leq \left( C \int_t^\infty Q(u)\,du \right) \int_0^t A(s)|y_2(s)| \, ds \leq C I^* \int_t^\infty Q(u)\,du \]
		Since $Q \in L^1(0, \infty)$, $\int_t^\infty Q(u)\,du \to 0$ as $t \to \infty$. Thus, $\tilde{R}_{\text{int}}(t) = o(1)$. Dividing $A(t)I(t)$ by $A(t)$ concludes the proof.
	\end{proof}
	
	\vspace{0.5cm}
	
	\begin{lemma}[Asymptotic Behavior of $J(t)$]
		Let $P_{4,B}(t,s) = - \frac{\partial}{\partial s} \left( \frac{1}{2}p(s) E(t, s) F(t, s) \right)$. Then $J(t) = \int_0^t P_{4,B}(t,s)y_2(s)\,ds$ satisfies:
		\[ J(t) = \frac{d_1 \cos(\omega t) + d_2 \sin(\omega t)}{A(t)} + o\left(\frac{1}{A(t)}\right) \]
	\end{lemma}
	
	\begin{proof}
		Factoring out $A(t)$ from $E(t,s)$, we have $P_{4,B}(t,s) = -\frac{1}{A(t)} \frac{\partial}{\partial s} \left( \frac{1}{2}p(s)A(s)F(t,s) \right)$. Because $\frac{1}{2}p(s)A(s) = A'(s)$ and $A''(s) = -q(s)A(s)$, expanding the derivative gives:
		\[ -\frac{\partial}{\partial s} (A'(s)F(t,s)) = -A''(s)F(t,s) - A'(s)F_s(t,s) = A(s) \left[ q(s)F(t,s) - \frac{1}{2}p(s)L_1(t,s) \right] \]
		where $F_s(t,s) = L_1(t,s) = \cos(\omega(t-s)) + \omega\sin(\omega(t-s))$ and $F(t,s) = \cos(\omega(t-s)) - \frac{1}{\omega}\sin(\omega(t-s))$. Grouping terms:
		\[ q(s)F(t,s) - \frac{1}{2}p(s)L_1(t,s) = \left( q(s) - \frac{p(s)}{2} \right)\cos(\omega(t-s)) - \left( \frac{q(s)}{\omega} + \frac{\omega p(s)}{2} \right)\sin(\omega(t-s)) \]
		We now apply the trigonometric addition formulas to separate $t$ and $s$, yielding:
		\begin{align*}
			q(s)F(t,s) - \frac{1}{2}p(s)L_1(t,s) &= \cos(\omega t) \underbrace{\left[ \left( q(s) - \frac{p(s)}{2} \right)\cos(\omega s) + \left( \frac{q(s)}{\omega} + \frac{\omega p(s)}{2} \right)\sin(\omega s) \right]}_{:= D_c(s)} \\
			&\quad + \sin(\omega t) \underbrace{\left[ \left( q(s) - \frac{p(s)}{2} \right)\sin(\omega s) - \left( \frac{q(s)}{\omega} + \frac{\omega p(s)}{2} \right)\cos(\omega s) \right]}_{:= D_s(s)}
		\end{align*}
		Therefore, $A(t)J(t) = \cos(\omega t) \int_0^t A(s)D_c(s)y_2(s)\,ds + \sin(\omega t) \int_0^t A(s)D_s(s)y_2(s)\,ds$.
		
		Since $p(s) \to 0$ and $q(s) \to 0$, the functions $D_c(s)$ and $D_s(s)$ are globally bounded by some constant $M_D$. Because $I^* < \infty$, the integrals converge absolutely:
		\[ d_1 = \int_0^\infty A(s)D_c(s)y_2(s)\,ds, \quad d_2 = \int_0^\infty A(s)D_s(s)y_2(s)\,ds \]
		The difference between the integral over $[0, t]$ and the limit is the tail $\int_t^\infty$, which satisfies:
		\[ \left| \int_t^\infty A(s)D_c(s)y_2(s)\,ds \right| \leq M_D \int_t^\infty A(s)|y_2(s)|\,ds \to 0 \quad \text{as } t \to \infty \]
		Thus, the tails are explicitly $o(1)$. Dividing $A(t)J(t) = d_1 \cos(\omega t) + d_2 \sin(\omega t) + o(1)$ by $A(t)$ completes the proof.
	\end{proof}
	
	\vspace{0.5cm}
	
	\begin{theorem}[Final Asymptotic Form of $P_4(t,s)$ Integration]
		\[ \int_0^t P_4(t,s)y_2(s)\,ds = \frac{e_1 \cos(\omega t) + e_2 \sin(\omega t)}{A(t)} + o\left(\frac{1}{A(t)}\right) \]
	\end{theorem}
	
	\begin{proof}
		By definition, $P_4(t,s) = P_{4,Q}(t,s) + P_{4,B}(t,s)$. By linearity of the integral:
		\[ \int_0^t P_4(t,s)y_2(s)\,ds = I(t) + J(t) \]
		Substituting the rigorously bounded asymptotic representations from Lemma 3 and Lemma 4:
		\[ I(t) + J(t) = \frac{c_1 \cos(\omega t) + c_2 \sin(\omega t)}{A(t)} + \frac{d_1 \cos(\omega t) + d_2 \sin(\omega t)}{A(t)} + o\left(\frac{1}{A(t)}\right) \]
		Factoring common terms establishes the theorem, where the final constants are explicitly $e_1 = c_1 + d_1$ and $e_2 = c_2 + d_2$.
	\end{proof}
	
\subsection{Asymptotic Preservation Result}
We can collect the results from the previous two sections to get the asymptotic behaviour of x.
\begin{theorem}
Suppose that $y_2(t)=o(1/A(t))$ as $t\to\infty$ and that $\int_0^\infty A(s)|y_2(s)|\,ds <+\infty$. Then there exist constants $c_1$ and $c_2$ such that 	
			\begin{equation} \label{eq:x_asymp_form}
		x(t) = \frac{c_1 \sin(\omega t) + c_2 \cos(\omega t)}{A(t)} + o\left(\frac{1}{A(t)}\right) \quad \text{as } t \to \infty,
	\end{equation}
\end{theorem}
\begin{proof}
We have from that \eqref{eq.decompxWDy2} that	\[
	x(t) = y_2(t)+\int_0^t E(t, s) L_3(t, s) y_2(s)\,ds + \int_0^t P_4(t,s)y_2(s)\,ds.
\]
We have shown in the last Theorem that
		\[ \int_0^t P_4(t,s)y_2(s)\,ds = \frac{e_1 \cos(\omega t) + e_2 \sin(\omega t)}{A(t)} + o\left(\frac{1}{A(t)}\right), \]
		and in Theorem we showed that  
		\[ 	\int_0^t E(t, s) L_3(t, s) y_2(s)\,ds=
		\frac{c_1 \cos(\omega t) + c_2 \sin(\omega t)}{A(t)} + o\left(\frac{1}{A(t)}\right).
		\]
		By hypothesis $y_2(t)=o(1/A(t))$ as $t\to\infty$, so combining this with the above asymptotic estimates, we have the desired result. 
\end{proof}
	
	\section*{Converse Convergence Result}
	
	We are given the Volterra integral equation:
	\[ y_2(t) = x(t) + \int_0^t K_3(t, s) x(s) \, ds \]
	where the kernel is defined by $K_3(t, s) = h''(t-s) + \omega^2 h(t-s) + p(s)h'(t-s) - p'(s)h(t-s)$, with $h(t) = t e^{-\omega^2 t}$. We are given the asymptotic relations $y_2(t) = o\left(\frac{1}{A(t)}\right)$ and $x(t) = \frac{c_1 \cos(\omega t) + c_2 \sin(\omega t)}{A(t)} + o\left(\frac{1}{A(t)}\right)$ as $t \to \infty$. We wish to prove that the integral $I_S(T) = \int_0^T A(t)\sin(\omega t) y_2(t)\,dt$ tends to a finite limit as $T \to \infty$ (and similarly for the cosine integral).
	
	\vspace{0.5cm}
	
	\begin{theorem}
		Under the given assumptions, the limits $$\lim_{T \to \infty} \int_0^T A(t)\sin(\omega t) y_2(t)\,dt$$ and $$\lim_{T \to \infty} \int_0^T A(t)\cos(\omega t) y_2(t)\,dt$$ exist and are finite.
	\end{theorem}
	
	\begin{proof}
		We detail the proof for the sine integral; the cosine case follows identically. Multiply the integral equation for $y_2(t)$ by $A(t)\sin(\omega t)$ and integrate from $0$ to $T$:
		\[ I_S(T) = \int_0^T A(t)\sin(\omega t) x(t)\,dt + \int_0^T A(t)\sin(\omega t) \left( \int_0^t K_3(t, s) x(s) \, ds \right) dt \]
		Since all functions are continuous on the compact domain $[0, T]$, we apply Fubini's theorem to swap the order of integration in the second term:
		\[ I_S(T) = \int_0^T A(s) x(s) \sin(\omega s)\,ds + \int_0^T x(s) \left( \int_s^T A(t)\sin(\omega t) K_3(t,s)\,dt \right) ds \]
		Let $J(s,T) = \int_s^T A(t)\sin(\omega t) K_3(t,s)\,dt$. We substitute the definition of $K_3(t,s)$ and separate it into two parts: $H(t-s) = h''(t-s) + \omega^2 h(t-s)$ and the terms involving $p(s), p'(s)$. 
		\[ J(s,T) = \int_s^T A(t)\sin(\omega t) H(t-s)\,dt + \int_s^T A(t)\sin(\omega t) [p(s)h'(t-s) - p'(s)h(t-s)]\,dt \]
		We integrate the first term by parts twice with respect to $t$. Using $h(0) = 0$ and $h'(0) = 1$, the first integration yields:
		\begin{align*}
			\int_s^T A(t)\sin(\omega t) h''(t-s)\,dt &= \left[ A(t)\sin(\omega t)h'(t-s) \right]_s^T - \int_s^T \frac{d}{dt}(A(t)\sin(\omega t)) h'(t-s)\,dt \\
			&= A(T)\sin(\omega T)h'(T-s) - A(s)\sin(\omega s) - \int_s^T \Phi'(t) h'(t-s)\,dt
		\end{align*}
		where we define $\Phi(t) = A(t)\sin(\omega t)$. Integrating the remaining integral by parts once more gives:
		\[ - \int_s^T \Phi'(t) h'(t-s)\,dt = -\Phi'(T)h(T-s) + \Phi'(s)h(0) + \int_s^T \Phi''(t)h(t-s)\,dt \]
		Thus, the integral over $H(t-s)$ becomes:
		\[ -A(s)\sin(\omega s) + A(T)\sin(\omega T)h'(T-s) - \Phi'(T)h(T-s) + \int_s^T [\Phi''(t) + \omega^2 \Phi(t)] h(t-s)\,dt \]
		Substituting $J(s,T)$ back into $I_S(T)$, we observe a crucial exact cancellation. The leading boundary term $-A(s)\sin(\omega s)$ exactly cancels the $\int_0^T A(s)x(s)\sin(\omega s)\,ds$ term outside the double integral. This leaves:
		\[ I_S(T) = O(T) + \int_0^T x(s) L(s,T)\,ds \]
		where $O(T)$ collects the boundary evaluations at $T$, and $L(s,T)$ collects the remaining integrals over $t$:
		\[ O(T) = \int_0^T x(s) \left[ A(T)\sin(\omega T)h'(T-s) - \Phi'(T)h(T-s) \right] ds \]
		\[ L(s,T) = \int_s^T [\Phi''(t) + \omega^2 \Phi(t)] h(t-s)\,dt + \int_s^T A(t)\sin(\omega t) [p(s)h'(t-s) - p'(s)h(t-s)]\,dt \]
		
		\textbf{Step 1: Asymptotic Behavior of the Oscillating Boundary Term $O(T)$} \\
		We must show that $O(T)$ tends to a finite constant as $T \to \infty$. Apply the change of variables $u = T-s$ to $O(T)$. We use the relations $x(T-u)A(T-u) = X(T-u)$ and $\lim_{T \to \infty} \frac{A(T)}{A(T-u)} = 1$. Since $\Phi'(T) = \omega A(T)\cos(\omega T) + o(A(T))$, applying the Dominated Convergence Theorem as $T \to \infty$ yields:
		\[ O(T) = \int_0^\infty X(T-u) [ \sin(\omega T)h'(u) - \omega \cos(\omega T)h(u) ]\,du + o(1) \]
		Substitute $X(T-u) = c_1 \cos(\omega(T-u)) + c_2 \sin(\omega(T-u))$. By expanding using trigonometric addition formulas, we can group $X(T-u)$ into components of $T$ and $u$:
		\[ X(T-u) = \cos(\omega T) C(u) + \sin(\omega T) S(u) \]
		where $C(u) = c_1\cos(\omega u) - c_2\sin(\omega u)$ and $S(u) = c_1\sin(\omega u) + c_2\cos(\omega u)$. Note the strict derivative relationships: $C'(u) = -\omega S(u)$ and $S'(u) = \omega C(u)$. Expanding the integrand of $O(T)$ produces three trigonometric coefficients relative to $T$:
		\begin{align*}
			O(T) &\approx \sin(\omega T)\cos(\omega T) \int_0^\infty [ C(u)h'(u) - \omega S(u)h(u) ]\,du \\
			&\quad + \sin^2(\omega T) \int_0^\infty S(u)h'(u)\,du - \cos^2(\omega T) \int_0^\infty \omega C(u)h(u)\,du
		\end{align*}
		For the cross term, observe that $C(u)h'(u) - \omega S(u)h(u) = C(u)h'(u) + C'(u)h(u) = \frac{d}{du}[C(u)h(u)]$. Its integral is $[C(u)h(u)]_0^\infty = 0$ because $h(0) = h(\infty) = 0$.
		For the squared terms, integrating $S(u)h'(u)$ by parts yields:
		\[ \int_0^\infty S(u)h'(u)\,du = [S(u)h(u)]_0^\infty - \int_0^\infty S'(u)h(u)\,du = -\int_0^\infty \omega C(u)h(u)\,du \]
		Let this shared integral value be denoted by $K$. Then the squared terms elegantly collapse:
		\[ \sin^2(\omega T) (K) + \cos^2(\omega T) (K) = K \]
		Thus, the oscillations perfectly annihilate one another, yielding $\lim_{T \to \infty} O(T) = K < \infty$.
		
		\textbf{Step 2: Convergence of the Integral Remainder} \\
		We must now guarantee that $\lim_{T \to \infty} \int_0^T x(s) L(s,T)\,ds$ converges. Expanding $$\Phi''(t) + \omega^2 \Phi(t) = A''(t)\sin(\omega t) + 2\omega A'(t)\cos(\omega t),$$ we replace $A'(t) = \frac{1}{2}p(t)A(t)$ and $A''(t) = \frac{1}{2}p'(t)A(t) + \frac{1}{4}p^2(t)A(t)$.
		The terms composing $L(s,T)$ contain functions of $p, p'$ and $p^2$. We first establish that $p' \in L^1(0, \infty)$. We are given $q(s) = -\frac{p'(s)}{2} - \frac{p^2(s)}{4} \in L^1(0, \infty)$. Rearranging and integrating gives $\int_0^t \frac{p^2(s)}{4}\,ds = -\int_0^t q(s)\,ds - \frac{p(t)}{2} + \frac{p(0)}{2}$. Since $p(t) \to 0$ and $q \in L^1$, the limit as $t \to \infty$ exists, proving $p^2 \in L^1(0, \infty)$. Consequently, $p'(s) = -2q(s) - \frac{p^2(s)}{2}$ is also strictly in $L^1(0, \infty)$.
		
		Because $p', p^2 \in L^1(0,\infty)$, any terms in $L(s,T)$ bounded by $A(t)(|p'(t)| + p^2(t))$ are absolutely integrable against $x(s) \sim X(s)/A(s)$. The singular threat to convergence is the term containing $p(s)$, as $p \notin L^1(0, \infty)$. 
		For any integral of the form $\int_0^\infty x(s) p(s) M(s)\,ds$ where $M(s)$ is bounded, we use the identity:
		\[ x(s)p(s) = \frac{X(s)}{A(s)}p(s) = -2 X(s) \frac{d}{ds}\left(\frac{1}{A(s)}\right) \]
		We integrate by parts:
		\[ \int_0^T -2 X(s) M(s) \frac{d}{ds}(A^{-1}(s))\,ds = \left[ -2 X(s) M(s) A^{-1}(s) \right]_0^T + \int_0^T \frac{1}{A(s)} \frac{d}{ds}(2 X(s) M(s))\,ds \]
		Since $A(T) \to \infty$, the boundary term tends to 0. Because $X(s)$ is a bounded trigonometric sum and $M(s)$ (composed of convolutions with $h$) has a bounded derivative, the integrand of the remaining term is $O(1/A(s))$. Because $A(s)$ grows monotonically, this integral converges by Dirichlet's test.
		
		Therefore, every component of the integral remainder converges as $T \to \infty$. Combined with $\lim_{T\to\infty} O(T) = K$, we conclude that $I_S(T)$ tends to a finite limit. 
	\end{proof}
	
				\bibliographystyle{unsrt}

\end{document}